\documentclass[12pt, a4paper, parskip=half, abstracton]{scrartcl}
\usepackage{graphicx} 
\title{Homological stability for automorphisms of symmetric bilinear forms}
\author{Vikram Nadig\thanks{Universit\"at Bielefeld\\vnadig@math.uni-bielefeld.de}}
\date{\today}
\usepackage{amssymb}
\usepackage{tikz-cd}
\usepackage{mathtools}
\usepackage{quiver}
\usepackage[hidelinks]{hyperref}
\usepackage[
backend=biber,
style=alphabetic,
]{biblatex}
\usepackage{amsthm}
\usepackage{relsize}
\usepackage{amsfonts}
\usepackage{amsmath}
\usepackage{yhmath}
\numberwithin{equation}{section}
\newtheorem{theorem}{Theorem}[section] 
\newtheorem{prop}[theorem]{Proposition}
\newtheorem{lem}[theorem]{Lemma}

\newtheorem{ddd}[theorem]{Definition}
\newtheorem{kor}[theorem]{Corollary}
\newtheorem{ass}[theorem]{Assumption}

\theoremstyle{remark}
\theoremstyle{definition}

\newtheorem{ex}[theorem]{Example}
\newtheorem{rem}[theorem]{Remark}
\begin{document}
\maketitle
\begin{abstract}
    We establish homological stability for automorphisms of symmetric bilinear forms over a class of principal ideal domains that includes all fields, the integers, the Gaussian integers, and the Eisenstein integers. In conjunction with Grothendieck-Witt theoretic calculations, this determines a large part of the stable cohomology of the odd orthogonal groups $O_{\langle g,g \rangle}(\mathbb Z)$ in low degrees.
\end{abstract}
\setcounter{tocdepth}{1}
\tableofcontents
\section{Introduction}
The cohomology of groups, and specifically of arithmetic groups, is a unifying theme within mathematics that has far-reaching applications to diverse areas like number theory, algebraic geometry, and $K$-theory. In recent times, homological stability has emerged as a nearly ubiquitous feature that these groups enjoy. From a calculational perspective, this has the consequence that knowledge of the stable (co)homology alone leads to the determination of many (co)homology groups. In practice, the stable (co)homology can be calculated by methods that are independent of those used to calculate any specific (co)homology groups. For instance, the stable (co)homology of general linear groups can be accessed through algebraic $K$-theory (\cite{origquill}), of mapping class groups and diffeomorphism groups of high dimensional even manifolds through Madsen-Tillmann spectra (\cite{mt1, madsenweiss,homstabmod,homstabmodtwo}), and of Higman-Thomson groups through Moore spectra (\cite{ht}). In all these instances, homological stability is known.

An important class of arithmetic groups are automorphisms of quadratic and symmetric bilinear forms (i.e. orthogonal groups) on finitely generated projective modules over a number ring. The stable cohomology thereof can be accessed through Grothendieck-Witt theory. To make this more precise, let $R$ be a ring and $\lambda \in \{s,q\}$ be one of two symbols. Let us say that an $s$-form is a symmetric bilinear form and that a $q$-form is a quadratic form. Let $\text{Unimod}^{\lambda}(R)^{\simeq}$ be the symmetric monoidal groupoid of non-degenerate (unimodular) $\lambda$-forms on finitely generated projective $R$-modules, with monoidal structure given by the direct sum. The Grothendieck-Witt theory of the pair $(R,\lambda)$ can be defined as the group completion $$\mathcal{GW}^{\lambda}(R) \coloneq (\text{Unimod}^{\lambda}(R)^{\simeq})^{\text{grp}}$$Let us say that an element $M$ of $\text{Unimod}^{\lambda}(R)^{\simeq}$ is cofinal if for any $N \in \text{Unimod}^{\lambda}(R)^{\simeq}$, there exists $N' \in \text{Unimod}^{\lambda}(R)^{\simeq}$ such that $N \oplus N' \cong M^{n}$ for some $n \in \mathbb N$. The group completion theorem (\cite{oggroup}) then implies that for a cofinal form $M$ $$H_*(\mathcal{GW}^{\lambda}(R)_0) \cong \text{colim}_{n} H_*(O(M^n))$$ where $O(M^n)$ is the group of isometries of $M^n$. The technological advances of the nine-author series of papers (\cite{nineuathorone}, \cite{nineauthortwo}, \cite{nineauthorthree}) in Hermitian $K$-theory have greatly improved the accessibility of the left-hand side above. It is therefore desirable to find many instances of cofinal forms $M$ whose orthogonal groups satisfy homological stability. The purpose of the present paper is to attack this problem for the form parameter $\lambda = s$ associated to symmetric bilinear forms.

 To put this into context, let us contrast this with the case $\lambda = q$, the form parameter associated to quadratic forms, where the situation is much better understood. For starters, it is a classical fact that the hyperbolic plane is always cofinal among all quadratic forms irrespective of the base ring $R$. Furthermore, homological stability for orthogonal groups of hyperbolic forms has been studied extensively in different generalities by Vogtmann (\cite{vogtmann}), Charney (\cite{charneytwo}), Panin (\cite{paninone} and \cite{panintwo}), Betley (\cite{betley}), Mirzaii - van der Kallen (\cite{unigrp}), and Friedrich (\cite{arboffset}). The most general of these results may be found in the work Mirzaii - van der Kallen and Friedrich, who treat rings with finite unitary stable rank. The former establishes homological stability for orthogonal groups of hyperbolic forms over these rings, and the latter improves upon the methods of the former to allow for an arbitrary non-zero offset form in the stability results.

 This is in sharp contrast to the case of symmetric bilinear forms. To the best of our knowledge, there are no homological stability results for the automorphism groups thereof in the existing literature, outside the case that the element $2 \in R$ is a unit. Of course, when $2$ is a unit, every symmetric bilinear form admits a unique quadratic refinement, and the problem reduces to the situation considered in the previous paragraph. In general, it is not clear to us as to how cofinal symmetric bilinear forms over a given ring $R$ can be determined. To illustrate the problem, using the classification theorem for non-degenerate indefinite forms over $\mathbb Z$ (namely, that they are classified by their rank, signature and parity - see \cite[Chp. 5, Sec. 2.2, Thm. 6]{arithmetic}) one can easily deduce that $$F \coloneqq \left(\mathbb Z^2, \begin{pmatrix}
1 & 0\\
0 & -1
\end{pmatrix}\right)$$ is cofinal. But this is not particularly robust: for instance, the form $$F' \coloneq \left(\mathbb Z[i]^2, \begin{pmatrix}
1 & 0\\
0 & -1
\end{pmatrix}\right)$$
is not cofinal among non-degenerate symmetric bilinear forms over $\mathbb Z[i]$ as, for any $n$, there exists no vector in $(F')^{n}$ whose inner product with itself equals $i$. Concerning the existence of cofinal forms, it is not difficult to see that every number ring admits a cofinal form (Proposition \ref{allnumbring}). However, it is not true that every Dedekind ring, or even every field, admits a cofinal form (as can be seen from Theorem \ref{cofincritintro} below).

Cofinal forms are intimately linked with metabolic forms. We recall that a form $M \in \text{Unimod}^{s}(R)^{\simeq}$ over a ring $R$ is \textit{metabolic} if it has a Lagrangian; that is, a subform $i:L \hookrightarrow M$ on an underlying projective submodule on which the inner product vanishes, and such that the null-composite sequence 
\[\begin{tikzcd}
	0 & L & {M \cong M^{*}} & {L^{*}} & 0
	\arrow[from=1-1, to=1-2]
	\arrow["i", from=1-2, to=1-3]
	\arrow["{i^{*}}", from=1-3, to=1-4]
	\arrow[from=1-4, to=1-5]
\end{tikzcd}\] is exact, where $(-)^{*}$ denotes the $R$-dual. For any $M \in \text{Unimod}^{s}(R)^{\simeq}$, the diagonal embedding exhibits the form $M \oplus (-M)$ as metabolic. The metabolic forms are therefore jointly cofinal. In particular, $R$ admits a cofinal form if and only if it admits a cofinal form that is metabolic. 

In the present paper, we give a reasonably satisfactory answer to both the algebraic problem of finding cofinal metabolic forms and the homotopical problem of establishing homological stability for the orthogonal groups thereof, for rings $R$ satisfying the following assumption:
\begin{ass}\label{onlyassint}
    $R$ is a principal ideal domain, and for each $r \in R$, either $r^2 \equiv 0 (\textup{mod } 2)$ or $r^2 \equiv u^2 (\textup{mod }2)$ for some unit $u \in R$.
\end{ass}
Let us briefly discuss examples and non-examples of rings satisfying this assumption. All fields are examples, and so are several number rings like $\mathbb Z$, $\mathbb Z[i]$, and $\mathbb Z[\omega]$. More generally, all quadratic integer rings of class number $1$ with discriminant $D \equiv 2,3 (\text{mod } 4)$ are examples, those with $D \equiv 5 (\text{mod }8)$ are examples if and only if the fundamental unit does not reduce to $1$ modulo $2$, and those with $D \equiv 1 (\text{mod }8)$ are non-examples (see Proposition \ref{whichnumberrings}). The assumption is quite restrictive in general, as, if $2$ is not a unit, either the ideal $(2)$ is prime or there is a unique prime ideal lying above $2$ with absolute ramification index two (Proposition \ref{twoprime}).

Our assumption is decisive in that it allows for a simple classification of metabolic forms over $R$. To explain this, let us additionally assume\footnote{This is not for technical reasons, but just to ease the present exposition. In our arguments, the case that $2$ is invertible is degenerate; and our results are new only when $2$ is not invertible.} for the moment that $2$ is not invertible in $R$. The set of squares $U(R)$ in $R/(2)$ is then a field, and $R/(2)$ is a vector space over $U(R)$. Given an element $(M,\lambda)$ of $\text{Unimod}^{s}(R)^{\simeq}$, we define its parity $P(M)$ to be the image of the set $\{\lambda(x,x)|x \in M\}$ in $R/(2)$. We show that metabolic forms are classified by their rank and parity (Theorem \ref{myclassification}). It follows relatively easily from this that:
\begin{theorem}\label{cofincritintro}
    Let $R$ be a ring satisfying Assumption \ref{onlyassint}. The following statements are equivalent:
    \begin{enumerate}
        \item There exists a cofinal form in $\textup{Unimod}^{s}(R)^{\simeq}$ that is metabolic.
        \item There exists a cofinal form in $\textup{Unimod}^{s}(R)^{\simeq}$.
        \item There exists a form $M$ in $\textup{Unimod}^{s}(R)^{\simeq}$ such that $P(M) = R/(2)$.
        \item $R/(2)$ is finite dimensional over $U(R)$.
    \end{enumerate}
    In this case, a metabolic form is cofinal if and only if its parity equals $R/(2)$.
\end{theorem}
We therefore find that a field of characteristic $2$ admits a cofinal form if and only if it is finite over its Frobenius. Thus the field $\mathbb F_2(t_1,...,t_n,...)$ is an example of one that admits no cofinal form. We also recover that the rank two metabolic form $\text{diag}(1,-1)$ is cofinal over $\mathbb Z$ and deduce that it remains cofinal after base-changing to $\mathbb Z[\omega]$. Over $\mathbb Z[i]$, we find that the rank four metabolic form $\text{diag}(1,1,i,i)$ is cofinal from which it follows that the rank two form $\text{diag}(1,i)$ is cofinal. Note, however, that $\text{diag}(1,i)$ is not metabolic. 

In light of the previous discussion, it is desirable to establish homological stability for orthogonal groups of metabolic forms. It is easy to see that every metabolic form is isomorphic to a direct sum of metabolic planes (Proposition \ref{met=tz}), and we show the following stability result:
\begin{theorem}\label{mainlintro}
Let $R$ be a ring satisfying Assumption \ref{onlyassint}. Suppose $M$ and $F$ are metabolic forms over $R$, with $\textup{rank}(F) = 2$. The map induced by $- \oplus F$ on homology with constant coefficients $$H_i(\textup{O}(M \oplus F)) \to H_i(\textup{O}(M \oplus F^2))$$ is an isomorphism for $i \leq f(\textup{rank}(M))$ where $f$ is a linear function of slope $1/4$.
\end{theorem}
For a metabolic cofinal form $M$, we can therefore deduce homological stability for the sequence $$O(M) \hookrightarrow O(M^2) \hookrightarrow ... \hookrightarrow O(M^n) \hookrightarrow ...$$

 We actually prove a stronger statement (Theorem \ref{mainl}) that allows us in specific instances to deduce homological stability for non-metabolic cofinal forms such as the previously discussed $\text{diag}(1,i)$ over $\mathbb Z[i]$. We do this using the machinery of Randal-Williams and Wahl (\cite{machine}) which also yields stability with respect to more general coefficient systems than just the constant ones. 

 Moreover, we furthermore deduce that any two metabolic planes induce the same isomorphism in their common stable range. In particular, this clarifies a point raised in \cite[Rem. 8.1.10]{hebestreit2025stablemodulispaceshermitian}.
 \begin{kor}\label{sameisointro}
 Let $R$ be a ring satisfying Assumption \ref{onlyassint}. Suppose $M$, $F$, and $F'$ are metabolic forms over $R$, with $\textup{rank}(F) = \textup{rank}(F') = 2$. There exists an isomorphism $\alpha: F \oplus F' \oplus F \cong F \oplus F' \oplus F'$ such that the triangle 
\[\begin{tikzcd}
	& {H_i(O(M \oplus F \oplus F'))} \\
	{H_i(O(M \oplus F \oplus F' \oplus F))} && {H_i(O(M \oplus F \oplus F' \oplus F'))}
	\arrow[from=1-2, to=2-1]
	\arrow[from=1-2, to=2-3]
	\arrow["{H_i(O(\textup{id}_M \oplus\alpha))}", from=2-1, to=2-3]
\end{tikzcd}\]
is a commutative diagram of isomorphisms for $i \leq f(\textup{rank}(M))$ where $f$ is a linear function of slope $1/4$. The left and right downwards arrows are stabilisations by $F$ and $F'$, respectively.
 \end{kor}

Coupling our results with calculations of Hebestreit-Steimle (\cite{hebestreit2025stablemodulispaceshermitian}) and upcoming work of Hebestreit-Land-Nikolaus (\cite{upcomingcalc}), we obtain the cohomology of the orthogonal group of direct sums of the previously mentioned cofinal form $\text{diag}(1,-1)$ over $\mathbb Z$ in low degrees, at regular primes:
\begin{kor}\label{calculationdintro}
    For an odd regular prime $p$, there are maps $$\mathbb F_2[v_i,w_i,a_i|i \geq 1] \to H^{*}(\textup{O}(\textup{diag}(1,-1)^{n}),\mathbb{F}_2)$$ $$\mathbb F_p[p_i^{+},p_i|i \geq 1] \otimes_{\mathbb F_p} \Lambda_{\mathbb F_p}[y_i|i \geq 1] \to H^{*}(\textup{O}(\textup{diag}(1,-1)^{n}),\mathbb{F}_p)$$ where the $v_i,w_i$ have degree $i$, the $p_i^{+}$ have degree $4i$, the $p_i$ have degree $2(p-1)i$, the $y_i$ have degree $(p-1)(2i-1)$, and the $a_i$ have degree $2i-1$, that are isomorphisms in degree $\leq \frac{n-9}{2}$. 
\end{kor}

Over $\mathbb Z[i]$, a metabolic plane $M$ is, up to isomorphism, one of the following four forms - the hyperbolic plane $H$, the forms represented by the diagonal matrices $\text{diag}(1,1)$ and $\text{diag}(i,i)$, and the form $$G \coloneq (\mathbb Z[i]^2, \begin{pmatrix}
0 & 1\\
1 & 1+i
\end{pmatrix}).$$
Theorem \ref{mainlintro} in particular yields homological stability for the orthogonal groups of $M^n$. We would therefore like to access the stable homology of these groups, although Theorem \ref{cofincritintro} implies that $M$ is not cofinal in $\text{Unimod}^{s}(\mathbb Z[i])^{\simeq}$. Of the four possibilities for $M$, the hyperbolic plane $H$ is the only one that admits a (unique) quadratic refinement and is therefore cofinal among \textit{even} symmetric bilinear forms over $\mathbb Z[i]$; its stable homology can therefore be accessed through even Grothendieck-Witt theory. It is not clear as to how the stable homology of the diagonal metabolic forms can be accessed. However, we can define a form parameter capturing the notion of ``half-even" forms - namely, those forms with parity $\{0,1+i\}$ - which interpolates between even symmetric bilinear forms and all symmetric bilinear forms; the form $G$ is then cofinal among the ``half-even" forms and the stable homology of the corresponding orthogonal groups is encoded in the Grothendieck-Witt theory associated to this form parameter. We intend to calculate the Grothendieck-Witt and $L$-theory of this form parameter in future work.

\textit{Acknowledgements: The author is very grateful to his advisor Fabian Hebestreit for proposing this problem, his extensive guidance, and his feedback on preliminary drafts. He also thanks Patrick Bieker, Julius Frank, Lennart Gehrmann, Manuel Hoff, Markus Land, and Nathalie Wahl for helpful discussions. Finally, he would like to thank Dan Petersen for pointing out a mistake in a talk based on a preliminary version of this article.}

\textit{The author was supported by the German Research Foundation (DFG) through the collaborative research centre “Integral structures in Geometry and Representation Theory” (grant no. TRR 358–491392403) at the University of Bielefeld.}

\section{Preliminaries}\label{two}
\subsection{Combinatorial ingredients}\label{twoone}
We set up notation and collect various homotopical facts on posets for later use, following \cite{genlin} and \cite{unigrp}.

Given a poset $(\mathcal{P},\leq)$, we can form a simplicial complex with vertices the elements of $\mathcal{P}$ with the rule that a collection of $k$ vertices forms a $k$-simplex if and only if they can be linearly ordered in $\mathcal{P}$. We let $|\mathcal{P}|$ be the geometric realization of this simplicial complex. The association $\mathcal{P} \to |\mathcal{P}|$ determines a functor from the category of posets to that of topological spaces. In what follows, we will mostly use this functor implicitly to associate topological concepts to posets.

Given a poset $\mathcal{P}$, $v \in \mathcal{P}$, and a subposet $\mathcal{S} \subseteq \mathcal{P}$, we define the poset $\text{Link}_\mathcal{S}(v) \coloneqq \text{Link}_{\mathcal{S}}^{-}(v) \cup \text{Link}^{+}_{\mathcal{S}}(v)$, where $\text{Link}_{\mathcal{S}}^{-}(v) = \{v' \in \mathcal{S}|v' \leq v,v' \neq v\}$ and $\text{Link}_{\mathcal{S}}^{+}(v) = \{v' \in \mathcal{S}|v \leq v',v' \neq v\}$. The significance of this definition lies in the following pushout square of topological spaces:

\[\begin{tikzcd}
	{|\text{Link}_{\mathcal{P}}(v)|} && {|\mathcal{P}-\{v\}|} \\
	{C(|\text{Link}_{\mathcal{P}}(v)|}) && {|\mathcal{P}|}
	\arrow[from=1-1, to=1-3]
	\arrow[from=1-1, to=2-1]
	\arrow[from=1-3, to=2-3]
	\arrow[from=2-1, to=2-3]
\end{tikzcd}\]
Here $C(-)$ denotes the cone.

Given a map of posets $f:\mathcal{P} \to \mathcal{Q}$ and $w \in \mathcal{Q}$, we define the fibre of $w$ along $f$ to be the poset $f/w \coloneqq \{v \in \mathcal{P}|f(v) \leq w\}$. By a height function on $\mathcal{P}$, we refer to a strictly increasing function $\mathcal{P} \to \mathbb Z_{\geq 0}$ (see \cite[Def. 3.3]{unigrp}). We will crucially use the following result to transfer acyclicity along suitable poset maps:
\begin{theorem}\label{3.6thing}
    Let $f: \mathcal{P} \to \mathcal{Q}$ be a map of posets, $ht$ be a height function on $\mathcal{Q}$ and $m$ be a positive integer such that for each element $v$ of $\mathcal{Q}$, the fibre $f/v$ is non-empty and $(ht(v)-m)$-acyclic. Further suppose that there exists an integer $n$ such that $\textup{Link}^{+}_{\mathcal{Q}}(v)$ is $(n- ht(v))$-acyclic for each element $v$ of $\mathcal{Q}$. Then the map $$H_{k}(f): H_k(\mathcal{P},\mathbb Z) \to H_k(\mathcal{Q} ,\mathbb Z)$$ is an isomorphism for $0 \leq k \leq n-m+2$.
\end{theorem}
\begin{proof}
    This is a mild generalisation of \cite[Thm. 3.6]{unigrp}, which is the case $m=1$, and can be proven in an identical fashion. In brief:

    We analyse the spectral sequence (see \cite[Thm. 3.1]{unigrp})$$E^2_{p,q} = H_p(\mathcal Q,v \mapsto H_q(f/v,\mathbb Z)) \implies H_{p+q}(\mathcal{P},\mathbb Z)$$ By assumption $H_q(f/v,\mathbb Z) = 0$ when $ht(v)\geq q+m$ and $q >0$, so \cite[Lem. 3.5]{unigrp} implies that $E^2_{p,q} =0$ when $p+q \leq n-m+2$ and $q > 0$. Consider the short exact sequence (which exists because we have assumed that $f/v$ is non-empty) $$0 \to \tilde{H}_0(f/v,\mathbb Z) \to H_0(f/v,\mathbb Z) \to \mathbb{Z} \to 0$$ By assumption, $\tilde H_0(f/v,\mathbb Z) = 0$ if $ht(v) \geq m$, so \cite[Lem. 3.5]{unigrp} yields that $H_p(\mathcal{Q},v \mapsto \tilde{H}_0(f/v,\mathbb Z)) = 0$ for $0 \leq p \leq n-m+2$. From the long exact sequence in homology associated to the short exact sequence of coefficients above, it follows that $E^2_{p,0} = H_p(\mathcal{Q},\mathbb Z)$ when $0 \leq p \leq n-m+2$.

    Therefore, $E^2_{p,q} \cong E^{\infty}_{p,q}$ when $0 \leq p+q \leq n-m+2$, from which the theorem follows.
\end{proof}
Let $X$ be a non-empty set. We define $\mathcal{O}(X)$ to be the poset of sequences in $X$, as in \cite[Def. 2.3]{genlin} - the elements are finite non-empty sequences $(x_1,...,x_k)$, $x_i \in X$, $x_i \neq x_j$ if $i \neq j$, with the order relation determined by $(x'_1,...,x'_{k'}) \leq (x_1,...,x_k)$ if there exists a strictly increasing map $f:\{1,...,k'\} \to \{1,...,k\}$ such that $x_{f(i)} = x'_i$ for $i = 1,...,k'$. For an element $x = (x_1,...,x_k)$ in $\mathcal{O}(X)$, we set $|x| = k$. The function $x \mapsto |x| - 1$ determines a height function on (any subposet of) $\mathcal{O}(X)$, which we shall simply denote by $ht$.

Let $\mathcal{P} \subseteq \mathcal{O}(X)$ be a subposet. For $k \in \mathbb N$, we will write $\mathcal{P}_{\geq k}$ and $\mathcal{P}_{\leq k}$ for the subposets of $P$ consisting of all sequences of length $\geq k$ and length $\leq k$, respectively. We shall say that $\mathcal{P}$ is downward closed if, for any $v \in \mathcal{P}$, and any $w \in \mathcal{O}(X)$, $w \leq v$ implies $w \in \mathcal{P}$. For an element $v = (v_1,...,v_k)$ in $\mathcal{P}$, we let $\mathcal {P}_v$ be the subposet of $\mathcal{P}$ consisting of those sequences $(w_1,...,w_l)$ such that $(w_1,...,w_l,v_1,...,v_k) \in \mathcal{P}$. Given a non-empty set $S$, we can form the poset $\mathcal{P}\langle S \rangle \subset \mathcal{O}(X \times S)$ consisting of those sequences $((x_1,s_1),...,(x_k,s_k))$ for which $(x_1,...,x_k) \in \mathcal{P}$. An element $s_0 \in S$ determines a map $l_{s_0}: \mathcal{P} \to \mathcal{P}\langle S \rangle$ by $(x_1,...,x_k) \mapsto ((x_1,s_0),...,(x_k,s_0))$. 
\begin{prop}\label{addset}
    Suppose $\mathcal{P}$ is downward closed and $n \in \mathbb{Z}$ has the property that $\mathcal{P}_v$ is $(n-|v|)$-connected for all $v \in V$.\begin{enumerate}
        \item The map $(l_{s_0})_{*}: H_k(\mathcal{P},\mathbb Z) \to H_k(\mathcal{P} \langle S \rangle, \mathbb Z)$ is an isomorphism for $0 \leq k \leq n$.
        \item If $\mathcal{P}$ is $\textup{min}\{1,n-1\}$-connected, then $\pi_k(\mathcal{P}) \to \pi_k(\mathcal{P}\langle S \rangle)$ is an isomorphism for $0 \leq k \leq n$.
    \end{enumerate}
\end{prop}
\begin{proof}
    This is \cite[Prop. 4.1]{unigrp}, which in turn follows from \cite[item 1.6]{classical} together with the fact that $l_{s_0}$ is a section of the projection map $\mathcal{P}\langle S \rangle \to \mathcal{P}$.
\end{proof}
We record a variant of the ``nerve theorem" for posets:
\begin{theorem}\label{nerve}
    Let $V$ and $T$ be non-empty sets, and let $\mathcal{P} \subseteq \mathcal{O}(V)$, $\mathcal{Q} \subseteq \mathcal{O}(T)$ be downward closed subposets. Let $v \mapsto \mathcal{Q}_v$ be an association, to each element $v$ of $\mathcal{P}$, a downward closed subposet $\mathcal{Q}_v \subseteq \mathcal{Q}$ such that $\mathcal{Q} = \bigcup_{v \in \mathcal{P}}\mathcal{Q}_v$ and $v \leq w \implies \mathcal{Q}_w \subseteq \mathcal{Q}_v$. Suppose that there is an integer $n$ such that:
    \begin{enumerate}
        \item For each element $v$ of $\mathcal{P}$, $\mathcal{Q}_v$ is $\text{min}\{n-1,n-|v|+1\}$-connected.
        \item For each element $x$ of $\mathcal{Q}$, the subposet $\mathcal{A}_x \coloneqq \{v \in \mathcal{P}|x \in \mathcal{Q}_v\}$ is $(n-|x|+1)$-connected.
        \item $\mathcal{P}$ is $n$-connected.
    \end{enumerate}
    Then $\mathcal{Q}$ is $(n-1)$-connected, and the natural map induced by the inclusions $$\bigoplus_{v \in \mathcal{P},|v|=1} H_n(\mathcal{Q}_v,\mathbb Z) \to H_n(\mathcal{Q},\mathbb Z)$$ is surjective. Moreover, if for each $v \in \mathcal{P}$ with $|v| = 1$, there exists an $n$-connected poset $\mathcal{Q}'_v$ with $\mathcal{Q}_v \subseteq \mathcal{Q}'_v \subseteq \mathcal{Q}$, then $\mathcal{Q}$ is $n$-connected.
\end{theorem}
\begin{proof}
    This is \cite[Thm. 4.7]{unigrp}.
\end{proof}
Let $M$ be a module over a ring $R$. We let $\mathcal{U}(M) \subset \mathcal{O}(M)$ be the poset of unimodular sequences - that is, the poset consisting of those sequences $(x_1,...,x_k)$ for which the submodule $\langle x_1,...,x_n\rangle \subseteq M$ is free on $x_1,...,x_n$ and is a direct summand of $M$. The following result is a fundamental ingredient in virtually all acyclicity arguments on subposets of $\mathcal{U}(M)$:
\begin{prop}\label{the2.12thingy}
    Suppose $M$ is a finitely generated free module. Let $\mathcal{P} \subseteq \mathcal{U}(M)$ be a downward closed subposet. For an element $v$ in $\mathcal{P}$, suppose there exists an integer $d$ such that for all $w \geq v$, the poset $\mathcal{P}_w$ is $(d-|w|)$-acyclic. Then $\textup{Link}^{+}_{\mathcal{P}}(v)$ is $(d-|v|)$-acyclic, and thus $\textup{Link}_{\mathcal{P}}(v)$ is $(d-1)$-acyclic.
\end{prop}
\begin{proof}
    See \cite[Prop. 2.12]{genlin} and its proof.
\end{proof}
We conclude this subsection with a technical fact that we shall use later.
\begin{prop}\label{coolmatroid}
    Let $V$ be a vector space of dimension $n$ over a field $\mathbf{k}$. Suppose $v_0,...,v_k \in V$ are $(k+1)$ vectors that span $V$. Let us say that a $(k-n)$-face $\sigma$ of the standard $k$-simplex $\Delta^{k}$ is good if $\{v_i|i \notin \sigma\}$ is a basis of $V$. The subcomplex $X_{V}$ of the $(k-n)$-skeleton of $V$ consisting of the good faces is $(k-n-1)$-connected.
\end{prop}
\begin{proof}
    Let us recall that a matroid $M$ is a simplicial complex $M$ with the additional property that for any simplices $A$ and $B$ of $M$ with $A$ having more elements than $B$, there exists a vertex $a \in A$ such that $B \cup \{a\}$ is also a simplex of $M$. It follows that $X_V$ is a matroid from \cite[Prop. 1.1.1, Thm. 2.1.1]{morematroid}. By \cite[Thm. 7.3.3]{matroidbook}, $X_V$ is shellable, and is therefore (\cite[Sec. 7.11.7]{matroidbook}) equivalent to a wedge sum of spheres of dimension $k-n$.
\end{proof}
\subsection{The homological stability machine}\label{twotwo}
In this subsection, we will briefly revisit the general approach to establishing homological stability developed in \cite{machine}.

Let $(\mathcal{C}, \oplus,0)$ be a symmetric\footnote{This can be done in the generality of a braided monoidal groupoid, as described in \cite{machine}, but we do not require this.} monoidal category in which $0$ is initial. Let $(A,X)$ be a pair of objects in $\mathcal{C}$.  
\begin{ddd}\textup{\cite[Def. 1.2]{machine}}
    We say that $\mathcal{C}$ is locally homogeneous at $(A,X)$ if the following two conditions hold:
    \begin{enumerate}
        \item For all $0 \leq p < n$, $\textup{Hom}_{\mathcal{C}}(X^{\oplus p +1}, A \oplus X^{\oplus n})$ is a transitive $\textup{Aut}_{\mathcal{C}}(A \oplus X^{\oplus n})$-set.
        \item For all $0 \leq p < n$, the map $\textup{Aut}_{\mathcal{C}}(A \oplus X^{\oplus n-p-1}) \to \textup{Aut}_{\mathcal{C}}(A \oplus X^{n})$ given by $f \mapsto f \oplus X^{\oplus p+1}$ is injective. Further, its image consists precisely of all those automorphisms that fix the map $$X^{\oplus p+1} \cong 0 \oplus X^{\oplus p+1} \to A \oplus X^{\oplus n-p-1} \oplus X^{\oplus p+1} \cong A \oplus X^{n}$$ induced by the initiality of $0$ and the identity on $X^{\oplus p+1}$.
        \end{enumerate}
\end{ddd}
 Suppose that $\mathcal{C}$ is locally homogeneous at $(A,X)$. Let $\mathcal{C}_{A,X}$ be the full subcategory of $\mathcal{C}$ whose objects are $A \oplus X^n$ for all $n \geq 0$. A coefficient system is a functor $\mathcal{A}: \mathcal{C}_{A,X} \to \text{Mod}(\mathbb Z)$. We refer to \cite[Def. 4.10]{machine} for the definitions of split coefficient systems and the degree of a general coefficient system.
 Given any pair of objects $(A,X)$ and for each $n$, let $W_n(A,X)$ be the semi-simplicial set with $p$-simplices $\text{Hom}_{\mathcal{C}}(X^{\oplus p+1},A \oplus X^{\oplus n})$, together with the obvious face maps. We will refer to $W_n(A,X)$ as the space of destabilisations associated to the pair $(A,X)$.

The main homological stability result we will use is the following:
\begin{theorem}\label{addlater}
    Let $\mathcal{C}$ be locally homogeneous at $(A,X)$, and let $\mathcal{A}$ be a coefficient system on $\mathcal{C}_{A,X}$ of degree $r$ at $0$. Suppose that for each $n$, $W_n(A,X)$ is $(n-2)/k$-connected and $k \geq 2$. Then the map induced by $ - \oplus X$ on homology with coefficients in $\mathcal{A}$ $$H_i(\textup{Aut}_{\mathcal{C}}(A \oplus X^{\oplus  n}),\mathcal{A}(A \oplus X^{\oplus n})) \to H_i(\textup{Aut}_{\mathcal{C}}(A \oplus X^{\oplus n+1}),\mathcal{A}(A \oplus X^{\oplus n+1}))$$ is an epimorphism if $i \leq n/k - r$ and an isomorphism if $i \leq (n-2)/k - r$. If $\mathcal{A}$ is split, then it is an epimorphism if $i \leq (n-r)/k$ and an isomorphism if $i \leq (n-r-2)/k$. If $\mathcal{A}$ is constant, then it is an epimorphism $i \leq n/k$ and an isomorphism if $i \leq (n-1)/k$.
\end{theorem}
\begin{proof}
    This follows from \cite[Thm. 3.1]{machine} and \cite[Thm. 4.20]{machine}.
\end{proof}
\begin{rem}\label{abeliancoefficient}
    Let $\mathcal{C}$ be locally homogeneous at $(A,X)$, and set $G_n \coloneq \text{Aut}_{\mathcal{C}}(A \oplus X^n)$. Put $G_{\infty} \coloneq \text{colim}_n G_n$, and let $G_{\infty}^{\text{ab}}$ be its abelianisation. An abelian coefficient system is a functor $\mathcal{A}: \mathcal{C}_{A,X} \to \text{Mod}( \mathbb Z[G_{\infty}^{\text{ab}}])$. The underlying $\mathbb Z$-modules $\mathcal{A}(A \oplus X^{n})$ acquire two commuting actions of $G_n$ through the functoriality of $\mathcal{A}$ and the $\mathbb Z[G_{\infty}^{\text{ab}}]$-module structure. Considering the diagonal action of $G_n$ on each $\mathcal{A}(A \oplus X^{n})$, Randal-Williams and Wahl also have a completely analogous homological stability result to Theorem \ref{addlater} for abelian coefficient systems that only requires the high-connectivity of the semi-simplicial sets $W_n(A,X)$, albeit with slightly worse bounds. We refer the interested reader to \cite[Thm. 3.4]{machine} and \cite[Thm. 4.20]{machine} for precise statements.
\end{rem}
We now recall a general construction made in \cite[Sec. 1.1]{machine} that produces locally homogeneous categories from groupoids, under mild assumptions. 

Let $(\mathcal{G}, \oplus, 0)$ be a symmetric monoidal groupoid. We do not assume that $0$ is initial anymore. Consider the category $U\mathcal{G}$ defined as follows. Its objects are the same as that of $\mathcal{G}$, while morphisms are given by $$\text{Hom}_{U\mathcal{G}}(A,B) = \text{colim}_\mathcal{G}\text{Hom}_{\mathcal{G}}(- \oplus A,B)$$

In more detail, this means that every morphism $A \to B$ in $U\mathcal{G}$ is specified by a pair $(X,f)$ with $X \in \mathcal{G}$ and $f: X \oplus A \to B$ a morphism in $\mathcal{G}$. Two such pairs $(X,f)$ and $(X',f')$ present the same morphism in $U\mathcal{G}$ if there is a morphism $g: X \to X'$ in $\mathcal{G}$ such that 
\[\begin{tikzcd}
	{X\oplus A} & B \\
	{X'\oplus A}
	\arrow["f", from=1-1, to=1-2]
	\arrow["{g\oplus A}"', from=1-1, to=2-1]
	\arrow["{f'}"', from=2-1, to=1-2]
\end{tikzcd}\] commutes. Composition is given by the formula $(Y,h) \circ (X,f) = (Y \oplus X, g \circ (X \oplus f))$.

In particular, we see that $0$ is initial in $U\mathcal{G}$.

There is an obvious functor $i:\mathcal{G} \to U \mathcal{G}$, which is the identity on objects and sends a morphism $f$ to $(0,f)$. 
\begin{prop}\label{ff}
    If $\textup{Aut}_{\mathcal G}(0) = \{\textup{id}\}$ and $\mathcal{G}$ has no zero divisors, then $i$ is fully faithful onto the core of $U\mathcal{G}$.
\end{prop}
\begin{proof}
    See \cite[Prop. 1.7]{machine}.
\end{proof}
We can endow $U\mathcal{G}$ with the symmetric monoidal structure constructed in the proof of \cite[Prop. 1.8]{machine}, which we shall denote by $\oplus_{\mathcal{G}}$. The object $0 \in U\mathcal{G}$ is the unit for $\oplus_{\mathcal{G}}$, and the functor $i: (\mathcal{G},\oplus,0) \to (U\mathcal{G},\oplus_{\mathcal{G}},0)$ is strong symmetric monoidal. 

We end with a criterion to check whether $U\mathcal{G}$ is locally homogeneous at a pair of objects:
\begin{prop}\label{crit}
    Let $(A,X)$ be a pair of objects in $\mathcal{G}$. Suppose the following two conditions hold:\begin{enumerate}
        \item For $0 \leq p < n$, if $Y \in \mathcal{G}$ satisfies $Y \oplus X^{\oplus p+1} \cong A \oplus X^{\oplus n}$ in $\mathcal{G}$, then $Y \cong A \oplus X^{\oplus n-p-1}$ in $\mathcal{G}$.
        \item For $0 \leq p < n$, the map $\textup{Aut}_{\mathcal{G}}(A \oplus X^{\oplus n - p -1}) \to \textup{Aut}_{\mathcal{G}}(A \oplus X^{\oplus n})$ given by $f \mapsto f \oplus X^{\oplus p + 1}$ is injective.
    \end{enumerate}
   Then $(U\mathcal{G},\oplus_{\mathcal{G}},0)$ is locally homogeneous at $(A,X)$.
\end{prop}
\begin{proof}
    See \cite[Prop. 1.10]{machine}.
\end{proof}
In practice, to establish homological stability for automorphism groups of $\mathcal{G}$, one does so for the automorphism groups of $U\mathcal{G}$ and uses Proposition \ref{ff} to check that these are the same.
\section{Symmetric bilinear forms and associated posets}\label{twothree}
In this section, we gather all the algebraic facts that we shall use in the sequel.

We begin with a quick review of some standard results. Much of this is basic and will be used subsequently without comment.

Recall that a symmetric bilinear form over a commutative ring $R$ is a pair $(M,\lambda)$ of an $R$-module $M$ and a function $\lambda: M \times M \to R$ that is $R$-linear in each argument and is invariant under the flip $C_2$-action on $M \times M$. The form $(M,\lambda)$ is non-degenerate if the map $M \to M^{*}$ given by $v \mapsto \lambda(v,-)$ is an isomorphism. If $N \subseteq M$ is a submodule such that $(N,\lambda|_{N})$ is non-degenerate, then the form $(M,\lambda)$ decomposes as $(M,\lambda) \cong (N,\lambda|_{N}) \oplus (N^{\perp},\lambda|_{N^{\perp}})$. In particular, in this case, $(M,\lambda)$ is non-degenerate if and only if $(N^{\perp},\lambda|_{N^{\perp}})$ is.

Suppose $N$ is a finitely generated free module and $\{e_1,...,e_n\}$ is a basis, then a symmetric bilinear form $\lambda'$ on $N$ is completely determined by the $(n \times n)$-matrix $\lambda'(e_i,e_j), 1 \leq i,j \leq n$. If $(M,\lambda)$ is a symmetric bilinear form and $(x_1,...,x_n)$ is a unimodular sequence (defined towards the end of Section \ref{twoone}) such that $\lambda(x_i,x_j) = \lambda'(e_i,e_j), 1 \leq i,j \leq n$, then $e_i \mapsto x_i$ determines an embedding $(N,\lambda') \hookrightarrow (M,\lambda)$. In this situation, if $(N,\lambda')$ and $(M,\lambda)$ are non-degenerate, and $M$ is finitely generated and free of the same rank as $N$, then the embedding $(N,\lambda') \hookrightarrow (M,\lambda)$ is an isomorphism because its orthogonal complement must be $0$.

Unimodular sequences in modules equipped with a non-degenerate symmetric bilinear form enjoy a special property:
\begin{prop}\label{asbasicasitgets}
    Let $(M,\lambda)$ be a non-degenerate symmetric bilinear form, and $(v_1,...,v_n)$ an element of $\mathcal{O}(M)$. Then $(v_1,...,v_n)$ is unimodular if and only if there exist $w_1,...,w_n$ in $M$ such that $\lambda(v_i,w_j) = \delta_{ij}$.
\end{prop}
\begin{proof}
    The proof of \cite[Lem. 3.3]{arboffset}, which deals with quadratic forms, goes through verbatim.
\end{proof}
\begin{kor}\label{abhayakara}
    Let $R$ be a principal ideal domain. A vector $v \in R^k$ is unimodular if and only if $\text{gcd}(\pi_1(v),...,\pi_k(v)) = 1$, where $\pi_i: R^k \to R$ is the $i^{th}$ projection.
\end{kor}
\begin{proof}
    This follows from applying Proposition \ref{asbasicasitgets} in the case $n=1$ to any non-degenerate form on $R^k$ (for example, the standard diagonal form $\text{diag}(1,...,1)$).
\end{proof}
\begin{kor}\label{subformunimod}
Let $(M,\lambda)$ be a non-degenerate symmetric bilinear form, and $(N,\lambda|_N) \subseteq (M,\lambda)$ a non-degenerate subform. Then $\mathcal{O}(N) \cap \mathcal{U}(M) = \mathcal{U}(N)$.
\end{kor}
\begin{proof}
    The inclusion $\mathcal{U}(N) \subseteq \mathcal{O}(N) \cap \mathcal{U}(M)$ is obvious. Conversely, given $(v_1,...,v_k) \in \mathcal{O}(N) \cap \mathcal{U}(M)$, Proposition \ref{asbasicasitgets} allows us to choose $w_1,...,w_n \in M$ be such that $\lambda(v_i,w_j) = \delta_{ij}$. The map $\lambda(w_i,-) \colon N \to R$ is an element of $N^{*}$, so by the non-degeneracy of $N$, we obtain $w_1',...,w_k' \in N$ such that $\lambda(v_i,w_j') = \delta_{ij}$, and the statement now follows by Proposition \ref{asbasicasitgets}.
\end{proof}
When no confusion can arise, we will simply write $M$ instead of $(M,\lambda)$. In the sequel, when discussing forms, unless specified otherwise, we will assume that the underlying module is finitely generated and projective.

We will restrict our study to rings $R$ satisfying the following property:
\begin{ass}\label{onlyass}
    $R$ is a principal ideal domain, and for each $r \in R$, either $r^2 \equiv 0 (\textup{mod } 2)$ or $r^2 \equiv u^2 (\textup{mod }2)$ for some unit $u \in R$.
\end{ass}
Examples of rings satisfying this assumption include all fields and several number rings of interest like $\mathbb Z, \mathbb Z[i]$, and $\mathbb Z[\omega]$. A general class of examples and non-examples is furnished by the following result:
\begin{prop}\label{whichnumberrings}
    Suppose $R$ is the ring of integers of $\mathbb Q(\sqrt D)$ for some square-free integer $D \neq 0,1$. Assume that $R$ is a principal ideal domain. Then $R$ satisfies Assumption \ref{onlyass} if and only if one of the following conditions hold:
    \begin{itemize}
        \item $D \equiv 2 (\textup{mod } 4)$
        \item $D \equiv 3 (\textup{mod } 4)$
        \item $D \equiv 5 (\textup{mod } 8)$ and the fundamental unit $u$ of $R$ satisfies $u \not \equiv 1 (\textup{mod }2)$.
    \end{itemize}
\end{prop}
\begin{proof}
    When $D \equiv 2 (\textup{mod } 4)$ or $D \equiv 3 (\textup{mod } 4)$, every element $x$ of $R$ can be expressed in the form $x = a  + b\sqrt{D}$ for $a,b \in \mathbb Z$. We find that $x^2 \equiv a^2 + b^2D (\text{mod }2)$. Since $a^2 + b^2D$ is an integer, Assumption \ref{onlyass} follows in this case.

    Suppose $D \equiv 1 (\textup{mod } 8)$. There exists an element $\alpha \in R$ satisfying $\alpha^2 - \alpha - (D-1)/4 = 0$. We therefore find that $\alpha^2 \equiv \alpha (\text{mod }2)$ and that that the image of $\alpha$ in $R/(2)$ is not a unit (because it is a zero divisor). Therefore $R$ does not satisfy Assumption \ref{onlyass} in this case.

    Lastly, suppose $D \equiv 5 (\textup{mod } 8)$. As before, there exists an element $\alpha \in R$ satisfying $\alpha^2 - \alpha - (D-1)/4 = 0$. Every element of $R$ can be expressed in the form $a + b\alpha$ for $a, b \in \mathbb Z$. Therefore, the composite $\{0,1,\alpha,\alpha+1\} \hookrightarrow R \to R/(2)$ is a bijection. Reducing the quadratic relation specifying $\alpha$ modulo $2$, we find that $\alpha^2 \equiv \alpha +1 (\text{mod }2)$. It follows that $R/(2) \cong \mathbb F_4$. We recall that every unit of $R$ is of the form $\pm u^{k}$, for $k \in \mathbb Z$, where $u$ is the fundamental unit. If $u \equiv 1 (\text{mod }2)$, it follows that the square of every unit in $R$ is congruent to $1$ modulo $2$, and Assumption \ref{onlyass} is not satisfied because $\alpha^2 \not \equiv 1 (\text{mod }2)$. If $u \not \equiv 1 (\text{mod }2)$, then the image of $u^2$ in the order $3$ cyclic group $R/(2)^{\times} \cong \mathbb F_4^{\times}$ must be a generator. Therefore, every non-zero element of $R/(2)$ comes from the image of a square unit in $R$, and Assumption \ref{onlyass} holds.
\end{proof}
 The following observation, which was pointed out to us by Manuel Hoff, shows that Assumption \ref{onlyass} is quite restrictive in general:
\begin{prop}\label{twoprime}
    Suppose $R$ is a ring that satisfies Assumption \ref{onlyass}. If $2 \in R$ is not a unit, then either $(2)$ is a prime ideal, or there exists a unique prime ideal $I$ that contains $(2)$, and it satisfies $I^2 = (2)$.
\end{prop}
\begin{proof}
    Suppose $2$ is not a unit and that $(2)$ is not prime. Let $I$ be a prime ideal containing $(2)$ and $x \in I$. The image of $x$ in $R/(2)$ is not a unit, and Assumption \ref{onlyass} thus forces $x^2 \in (2)$. If $J$ is a possibly different prime ideal that contains $(2)$, for $y \in J$, we find that $y^2 \in (2) \subset I$, so $y \in I$ and thus $I = J$. The existence of prime factorisations shows that $I^2 = (2)$.
\end{proof}
Before examining the ramifications of Assumption \ref{onlyass}, we make some general observations about metabolic forms (which were defined in the introduction).

 For a ring $R$, let $\text{Met}_2(R)$ be the set of (isomorphism classes of) non-degenerate symmetric bilinear forms on $R^2$ for which there is a unimodular vector $x \in  R^2$ such that $\lambda(x,x) = 0$. 

 \begin{lem}\label{met2freelag}
     $\textup{Met}_2(R)$ is precisely the set of metabolic forms of rank two with free Lagrangian.
 \end{lem}
 \begin{proof}
     Suppose $M$ is a rank two metabolic form with free Lagrangian $L$. As modules, we have $M \cong L \oplus L^{*}$, so any generator $x$ of $L$ is necessarily unimodular and satisfies $\lambda(x,x) = 0$. Conversely, let $M \in \text{Met}_2(R)$, and choose a unimodular vector $x$ such that $\lambda(x,x) = 0$. Put $L \coloneq \langle x \rangle$. We will show that \[\begin{tikzcd}
	0 & L & {M \cong M^{*}} & {L^{*}} & 0
	\arrow[from=1-1, to=1-2]
	\arrow[, from=1-2, to=1-3]
	\arrow[, from=1-3, to=1-4]
	\arrow[from=1-4, to=1-5]
\end{tikzcd}\] is exact. By definition, we see that the sequence is null-composite and exact at $L$. Because $x$ is unimodular, by Proposition \ref{asbasicasitgets}, there exists $y \in M$ such that $\lambda(x,y) = 1$. We find that the restriction of the form to the subspace $\langle x,y \rangle$ is given by the matrix $$\begin{pmatrix}
0 & 1\\
1 & \lambda(y,y)
\end{pmatrix}$$ with respect to the basis $\{x,y\}$. Therefore, we have a non-degenerate embedding of a rank two form into $M$, which shows that $\langle x,y \rangle = M$. In particular, $L$ is a direct summand of $M$, so we have exactness at $L^{*}$. Finally, if $v = ax + by \in M$ satisfies $\lambda(v,x) = 0$, then $b = 0$; this shows exactness at $M$ and completes the proof.
 \end{proof}
Consider the function $$\theta\colon R \to \text{Met}_2(R), r \mapsto \begin{pmatrix}
0 & 1\\
1 & r
\end{pmatrix}$$
The vector $(1,0)$ is unimodular and isotropic in $\theta(r)$ with respect to these coordinates.
Observe that $\theta(0)$ is the (isomorphism class of the) hyperbolic plane, which we shall denote by $H$.
\begin{prop}\label{thetasurj}
    The map $\theta: R \to \textup{Met}_2(R)$ descends to a surjective map $\theta: R/(2) \to \textup{Met}_2(R)$.
\end{prop}
\begin{proof}
    We have seen that $\theta: R \to \text{Met}_2(R)$ is surjective in the proof of Lemma \ref{met2freelag}. Suppose $r,s \in R$ are such that $r - s = 2t$, for some $t \in R$. Then the map $\theta_t\colon \theta(r) \to \theta(s)$ given in standard coordinates by $(x_1,x_2) \mapsto (x_1 +tx_2,x_2)$ is an isomorphism of symmetric bilinear forms.
\end{proof}
The following result generalises Lemma \ref{met2freelag}:
 \begin{prop}\label{met=tz}
     A non-degenerate symmetric bilinear form is metabolic with free Lagrangian if and only if it is isomorphic to a (finite) direct sum of forms in $\textup{Met}_2(R)$.
 \end{prop}
 \begin{proof}
     Let us show the ``if" direction first. A direct sum of metabolic forms is metabolic, so it suffices to observe that each element of $\text{Met}_2(R)$ is metabolic. This follows from Lemma \ref{met2freelag}.

Turning to the converse, suppose $M$ is metabolic with Lagrangian $L \hookrightarrow M$; we have the exact sequence\[\begin{tikzcd}
	0 & L & {M \cong M^{*}} & {L^{*}} & 0
	\arrow[from=1-1, to=1-2]
	\arrow[, from=1-2, to=1-3]
	\arrow[, from=1-3, to=1-4]
	\arrow[from=1-4, to=1-5]
\end{tikzcd}\]
Assume that $L$ is free on some $\{x_1,...,x_n\}$; we will proceed by induction on $n$. By definition of $L$, $\lambda(x_i,x_i) = 0$ for each $i$. We observe that $\text{rank}(M) = 2n$. When $n = 1$, consider the element in $L^{*}$ that sends $x_1$ to $1 \in R$. By exactness of the sequence, this gives $y_1 \in M$ such that $\lambda(x_1,y_1) = 1$. Note that $x_1$ is therefore unimodular by Proposition \ref{asbasicasitgets}. As in Lemma \ref{met2freelag}, we conclude that $M = \langle x,y \rangle \in \text{Met}_2(R)$. For the inductive step, we choose $y_1 \in M$ to represent the element of $L^{*}$ that maps $x_i$ to $\delta_{1i} \in R$. Let $M'$ be the orthogonal complement of $\langle x_1,y_1 \rangle$ in $M$; this is a non-degenerate subform. The given exact sequence splits as the direct sum of the exact sequences \[\begin{tikzcd}
	0 & {\langle x_1 \rangle} & {\langle x_1,y_1 \rangle} & {\langle x_1 \rangle ^{*}} & 0
	\arrow[from=1-1, to=1-2]
	\arrow[, from=1-2, to=1-3]
	\arrow[, from=1-3, to=1-4]
	\arrow[from=1-4, to=1-5]
\end{tikzcd}\]
\[\begin{tikzcd}
	0 & {\langle x_2,...,x_n \rangle} & {M' \cong M'^{*}} & {\langle x_2,...,x_n \rangle ^{*}} & 0
	\arrow[from=1-1, to=1-2]
	\arrow[, from=1-2, to=1-3]
	\arrow[, from=1-3, to=1-4]
	\arrow[from=1-4, to=1-5]
\end{tikzcd}\] and the statement follows from the inductive hypothesis, as the lower exact sequence exhibits $M'$ as metabolic with a Lagrangian that is free on $n-1$ generators.
 \end{proof}
Let us now restrict to rings $R$ that satisfy Assumption \ref{onlyass}; this assumption will be in force for the rest of this paper.
Let $S(R)$ be the set of equivalences classes of $R/(2)$ under the equivalence relation $a \sim b$, for $a,b \in R/(2)$, if there is a unit $u \in R$ such that $a = v^2b$, where $v$ is the image of $u$ in $R/(2)$.
\begin{prop}\label{tzr}
The surjective map $\theta: R/(2) \to \textup{Met}_2(R)$ of Proposition \ref{thetasurj} induces a bijection $S(R) \cong \textup{Met}_2(R)$.
\end{prop}
\begin{proof}
     For any unit $u \in R$, and any element $r \in R$, the map $\theta_u\colon \theta(r) \to \theta(u^2r)$ given in standard coordinates by $(x_1,x_2) \mapsto (x_1,x_2u^{-1})$ is evidently an isomorphism. Thus, $\theta$ descends to a surjective map $S(R) \to \text{Met}_2(R)$.

It remains to check that $\theta$ is injective. To this end, suppose $\theta(r) = \theta(s)$ for $r,s \in R$. Then there exists a vector $(x_1,x_2)$ such that $2x_1x_2 + rx_2^2 = s$. Considering this equation modulo $2$ and using Assumption \ref{onlyass}, it follows that $r$ and $s$ have the same image in $S(R)$ except possibly when $s \equiv 0 (\text{mod }2)$. In this case, using the existence of a vector $(y_1,y_2)$ such that $2y_1y_2 + sy_2^2 = r$, we see that $r \equiv 0 (\text{mod }2)$ as well. This completes the proof.
\end{proof}
We would like to classify metabolic forms. To do this, we introduce the following notion:
\begin{ddd}\label{parity}
    Let $M$ be a non-degenerate symmetric bilinear form. Consider the subset $I(M) \coloneq\{\lambda(x,x)|x \in M\} \subseteq R$. We define the parity of $M$, $P(M)$, as the image of $I(M)$ in $R/(2)$.
\end{ddd}
 In particular, the parity of the zero form is $\{0\}$.

 Let $U(R)$ be the set of squares in $R/(2)$.
 \begin{prop}\label{essential}
     Suppose $2$ is not a unit in $R$. Then $U(R)$ is a subfield of the ring $R/(2)$, and for any non-degenerate symmetric bilinear form $M$, $P(M)$ acquires the structure of a vector space over $U(R)$. Furthermore, $P(M)$ is finite dimensional over $U(R)$ and its dimension is at most $\textup{rank}(M)$.
 \end{prop}
 \begin{proof}
     The assertion that $U(R)$ is a field is immediate from Assumption \ref{onlyass}. The vector space structure also follows from Assumption \ref{onlyass} and the observations that $\lambda(x,x) + \lambda(y,y) \equiv \lambda(x+y,x+y) (\text{mod }2)$ and $\lambda(ux,ux) = u^2\lambda(x,x)$ for $x,y \in M$ and a unit $u \in R$. The second statement follows from the observation that for any basis $\{e_1,...,e_{\text{rank}(M)}\}$ of $M$, the images of the $\lambda(e_i,e_i)$ in $R/(2)$ generate $P(M)$ as a $U(R)$-vector space. Indeed, for $x = \Sigma_{i=1}^{\text{rank}(M)}x_ie_i$, we notice that $\lambda(x,x) \equiv \Sigma_{i=1}^{\text{rank}(M)}x_i^2 \lambda(e_i,e_i) (\text{mod } 2)$ and the assertion then follows from Assumption \ref{onlyass}.
 \end{proof}

\begin{lem}\label{exampleparity}
   For forms $M$ and $M'$, we have $P(M \oplus M') = P(M) + P(M')$. In particular, $P(M^n) = P(M)$ for all $n \geq 1$.  
\end{lem}
\begin{proof}
   Obvious.
\end{proof}

In what follows, we will sometimes view $\theta$ as a function valued in forms, and not in isomorphism classes thereof.
\begin{lem}\label{addition}
    Let $r,s \in R$. Then $\theta(r) \oplus \theta(r+s) \cong \theta(r) \oplus \theta(s)$.
\end{lem}
\begin{proof}
    By checking on the standard basis, it is evident that the map $\psi: \theta(r) \oplus \theta(r+s) \to \theta(r) \oplus \theta(s)$ given in standard coordinates by $((x_1,x_2),(y_1,y_2)) \mapsto ((x_1,x_2 + y_2),(y_1-x_1-rx_2,y_2))$ is an isomorphism of forms.
\end{proof}
\begin{lem}\label{standardparity}
    Let $r_1,...,r_n \in R$. The image of the set $\{\Sigma_{i=1}^{n}u_i^2r_i|u_i \in R^{\times} \cup \{0\}\}$ under the quotient map $R \to R/(2)$ is $P(\oplus_{i=1}^{n}\theta(r_i))$.
\end{lem}
\begin{proof}
    By definition, $I(\oplus_{i=1}^{n}\theta(r_i)) = \{2x_1x_2 + r_1x_2^2 + ... + 2x_{2n-1}x_{2n} + r_nx_{2n}^2|(x_1,...,x_{2n}) \in R^{2n}\}$. The assertion immediately follows from Assumption \ref{onlyass}.
\end{proof}
\begin{theorem}\label{myclassification}
    Let $M$ and $M'$ be metabolic forms. Then $M \cong M'$ if and only if $\textup{rank}(M) = \textup{rank}(M')$ and $P(M) = P(M')$.
\end{theorem}
\begin{proof}
    The ``only if" direction is obvious. Let us turn to the converse.
    
     By Propositions \ref{thetasurj} and \ref{met=tz}, we can write $M \cong \oplus_{i=1}^{n} \theta(r_i)$ and $M' \cong \oplus_{i=1}^{n} \theta(r'_i)$, for $r_i,r'_i \in R$. Assume without loss of generality that $r_i = r'_i$ for $i \leq k$ and $r_i \neq r'_i$ for $i > k$, for some $k \in \{0,...,n-1\}$. We will show that the sequences $(r_1,...,r_n)$ and $ (r'_1,...,r'_n)$ can then be modified to sequences $(s_1,...,s_n)$ and $(s_1',...,s_n')$, respectively, such that $M \cong \oplus_{i=1}^{n} \theta(s_i)$, $M' \cong \oplus_{i=1}^{n} \theta(s'_i)$, and $s_i = s'_i$ if $i \leq k+1$.

    Let us delineate two cases. \begin{itemize}
        \item Suppose $P(\oplus_{i=1}^{k}\theta(r_i)) = P(M) = P(M')$. By Lemma \ref{standardparity}, there exists a subset $J \subset \{1,...,k\}$ and for each $j \in J$, a unit $u_j \in R$ such that $\Sigma_{j \in J} u_{j}^2r_j \equiv r_{k+1}(\text{mod }2)$. If $J$ is not empty, using $\theta(r_j) = \theta(u_{j}^2r_j)$, Proposition \ref{tzr} and Lemma \ref{addition} (repeatedly), we see that $\oplus_{i=1}^{k+1}\theta(r_i) \cong \oplus_{i=1}^{k}\theta(r_i) \oplus \theta(0)$. This decomposition also clearly holds when $J$ is empty. Similarly, $\oplus_{i=1}^{k+1}\theta(r'_i) \cong \oplus_{i=1}^{k}\theta(r'_i) \oplus \theta(0)$, and we are done.
        \item Else, let us assume without loss of generality that the image of $r'_{k+1}$ in $R/(2)$ does not belong to $P(\oplus_{i=1}^{k}\theta(r_i))$. Since $P(M) = P(M')$, it follows from Lemma \ref{standardparity} that there exists a subset $J \subset \{1,...,n\}$ such that $J' \coloneq J \cap \{k+1,...,n\}$ is non-empty, and units $u_j \in R$ for each $j \in J$ such that $\Sigma_{j \in J}u_j^2r_j \equiv r'_{k+1}(\text{mod }2)$. Assume without loss of generality that $k+1 \in J \cap \{k+1,...,n\}$. As in the previous case, using $\theta(r_j) = \theta(u_{j}^2r_j)$, Proposition \ref{tzr} and Lemma \ref{addition} (repeatedly), we can replace $r'_{k+1}$ by $\Sigma_{j \in J'}u_{j}^2r_j$. By the same reasoning together with our assumption that $k+1 \in J'$, we can also replace $r_{k+1}$ by $\Sigma_{j \in J'}u_{j}^2r_j$. Therefore we have arranged for $r_i = r'_i$ if $i \leq k+1$, and this concludes the argument.
    \end{itemize}
\end{proof}
\begin{ex}\label{classfail}
    Let us illustrate that the above rank-parity classification of metabolic forms does not hold in general for rings, and even principal ideal domains, that do not satisfy Assumption \ref{onlyass}. We note that the definition of parity (Definition \ref{parity}) makes sense for forms defined over any ring. Let $S$ be the ring of integers of $\mathbb Q(\sqrt {37})$. It is known that $S$ is a PID and that its fundamental unit is $6 + \sqrt{37}$. We see from Proposition \ref{whichnumberrings} that $S$ does not satisfy Assumption \ref{onlyass}. Moreover, from the proof of Proposition \ref{whichnumberrings}, we learn that $S/(2) \cong \mathbb F_4$ and that the set of squares in $S$ surjects onto $S/(2)$, while the square units are all congruent to $1$ modulo $2$. Let $\alpha = \frac{1 + \sqrt{37}}{2} \in S$. These observations show that the metabolic planes over $S$ given by $$\begin{pmatrix}
0 & 1\\
1 & 1
\end{pmatrix} \text{  and  } \begin{pmatrix}
0 & 1\\
1 & \alpha
\end{pmatrix}$$ have the same parity. We claim that these are not isomorphic. Indeed, if they were, there would exist $x_1,x_2 \in S$ with $2x_1x_2 + \alpha x_2^2 = 1$. This relation forces $x_2$ to be a unit, so $x_2^{2} \equiv 1 (\text{mod }2)$. Reducing the previous relation modulo $2$ however yields that $\alpha x_2^{2} \equiv 1 (\text{mod }2)$, and we have arrived at a contradiction because $\alpha \not \equiv 1 (\text{mod }2)$.
\end{ex}

We now introduce and study combinatorial properties of various posets that arise as algebraically defined subposets of $\mathcal{O}(M)$ or $\mathcal{O}(M \times M)$, where $M$ is a non-degenerate symmetric bilinear form over $R$.

We start by recalling the poset of unimodular sequences $\mathcal{U}(M)$ defined at the end of Section \ref{twoone}. This is the key player in van der Kallen's proof of homological stability for general linear groups (\cite{genlin}). However, in our situation, we will need to carefully examine the subposet $\mathcal{IU}(M) \subseteq \mathcal{U}(M)$ of \textit{isotropic} unimodular sequences, namely, of sequences $(x_1,...,x_k)$ such that $\lambda(x_i,x_j) = 0$ for all $1 \leq i,j \leq k$. Also of relevance is the poset $\mathcal{U}'(M)$ consisting of unimodular sequences $(x_1,...,x_k)$ satisfying $\lambda(x_i,x_i) = 0$ for all $1 \leq i \leq k$. We have inclusions $\mathcal{IU}(M) \subseteq \mathcal{U}'(M) \subseteq \mathcal{U}(M)$. Our nomenclature here is chosen so as to be consistent with that of \cite{unigrp} for the analogous posets in the quadratic case.

We will crucially use the following notion to associate a meaningful notion of parity to isotropic unimodular sequences.
\begin{ddd}
    Let $(v_1,...,v_k) \in \mathcal{IU}(M)$. A sequence $(w_1,...,w_k) \in \mathcal{O}(M)$ is called a witnessing sequence to $(v_1,...,v_k)$ if $\lambda(v_i,w_j) = \delta_{ij}$, and $\lambda(w_i,w_j) = 0$ when $i \neq j$, for $1 \leq i,j \leq n$.
\end{ddd}
Note, by Proposition \ref{asbasicasitgets}, that witnessing sequences are necessarily unimodular.

A witnessing sequence $(w_1,...,w_k)$ to $(v_1,...,v_k)$ determines an embedding $$e_{v,w}:T_1 \oplus ... \oplus T_k \to N$$ where $T_i = \langle v_i,w_i \rangle$. Note that $T_i$ is isomorphic to $\theta(\lambda(w_i,w_i))$. We will abusively use $e_{v,w}$ for both the embedding and its image when the context is clear. As a preliminary observation, we note that the isomorphism type of the orthogonal complement $e_{v,w}^{\perp}$ does not depend on the choice of witnessing sequence $(w_1,...,w_k)$:
\begin{lem}\label{nochoicewitness}
    Let $(v_1,...,v_k) \in \mathcal{IU}(M)$, and suppose that $(w_1,...,w_k)$ and $(w_1',...,w_k')$ are witnessing sequences to it. Then $e_{v,w}^{\perp} \cong e_{v,w'}^{\perp}$.
\end{lem}
\begin{proof}
    Choose a basis $\{a_1,...,a_n\}$ of $e_{v,w}^{\perp}$. The elements $a'_i \coloneqq
    a_i - \sum_{j=1}^{k} \lambda(a_i,w'_j)v_j$, $i=1,...,n$, belong to $e_{v,w'}^{\perp}$ and satisfy $\lambda(a'_i,a'_j) = \lambda(a_i,a_j)$, thereby furnishing a map $e_{v,w}^{\perp} \to e_{v,w'}^{\perp}$, which is an isomorphism because $\text{rank}(e_{v,w}^{\perp}) = \text{rank}(e_{v,w'}^{\perp})$.
\end{proof}
Before we do anything with witnessing sequences, let us establish their existence.
\begin{prop}\label{theyexist}
    Let $(v_1,...,v_k) \in \mathcal{IU}(M)$. There exists a witnessing sequence to $(v_1,...,v_k)$.
\end{prop}
\begin{proof}
    When $k = 1$, Proposition \ref{asbasicasitgets} furnishes a vector $w_1$ such that $\lambda(v_1,w_1) = 1$. Then $(w_1)$ is a witnessing sequence to $(v_1)$. For larger $k$, we first choose $w_1$ such that $\lambda(v_i,w_1) = \delta_{i1}$ using Proposition \ref{asbasicasitgets}, and then proceed inductively on the orthogonal complement of $\langle v_1,w_1 \rangle$. Note that this is the same argument that appeared in the proof of the ``only if" part of Proposition \ref{met=tz}.
\end{proof}
We can use witnessing sequences to derive elementary structural results on the poset $\mathcal{IU}(M)$.
\begin{lem}\label{sequenceext}
    Suppose $n \in \mathbb N_{>0}$ is the maximum length of a sequence in $\mathcal{IU}(M)$. Suppose $(v_1,...,v_k) \in \mathcal{IU}(M)$, with $k < n$. Then there exists $(v_1',...,v'_n) \in \mathcal{IU}(M)$ such that $(v_1,...,v_k) \leq (v_1',...,v_n')$.
\end{lem}
\begin{proof}
    We will show that any sequence in $\mathcal{IU}(M)$ of length $k < n$ can be extended to a sequence of length $k+1$. 
    
    Let us first assume that $R$ is a field. We will proceed by induction on $k$. Only for the purposes of this argument, it will be convenient to add a minimum element to $\mathcal{IU}(M)$ and declare it to be of length zero. The base case $k = 0$ is then satisfied by construction.

    Now consider, for $k \geq 1$, a sequence $(v_1,...,v_k) \in \mathcal{IU}(M)$. As a preliminary comment, we note that for any basis $u_1,...,u_k$ of $V \coloneq \langle v_1,...,v_k \rangle$, the sequence $(u_1,...,u_k) \in \mathcal{IU}(M)$, and it can be extended to a longer one if and only if the same holds for $(v_1,...,v_k)$. As in the base case, let $(x_1,...,x_n) \in \mathcal{IU}(M)$ be some maximal length sequence. Choose a witnessing sequence $(w_1,...,w_k)$ to $(v_1,...,v_k)$. Let $\pi: M \to e_{v,w}$ be the projection, and set $x'_i \coloneq \pi(x_i)$. There exist $a_1,...,a_n,b_1,...,b_k \in R$ such that not all the $a_1,...,a_n$ equal $0$ and $\Sigma_{i=1}^{n}a_ix'_i = \Sigma_{j=1}^{k}b_jv_j$ (if some $x'_i = 0$, we can take $a_i = 1$ and $a_j = b_l = 0$ for all $j,l$ with $j \neq i$; else the $x'_1,...,x_n'$ and the $v_1,...,v_k$ are $n + k > 2k$ non-zero vectors in the $2k$-dimensional subspace $e_{v,w}$ and there is a non-trivial dependence relation). Put $y \coloneq \Sigma_{i=1}^{n}a_ix_i - \Sigma_{j=1}^{k}b_jv_j$; then $y \in e_{v,w}^{\perp}$ and $\lambda(y,y) = 0$. If $y \neq 0$, then $(v_1,...,v_k,y) \in \mathcal{IU}(M)$ and we are done. Now suppose $y = 0$. By the unimodularity of $(x_1,...,x_n)$, $u_1 \coloneq \Sigma_{j=1}^{k}b_jv_j = \Sigma_{i=1}^{n}a_ix_i$ is non-zero. Assume without loss of generality that $a_1 \neq 0$. Then $u_1,x_2,...,x_n$ is a basis of $X \coloneq \langle x_1,...,x_n \rangle$, so $(u_1,x_2,...,x_n) \in \mathcal{IU}(M)$. Extend $u_1$ to a basis $u_1,...,u_k$ of $V$. By the comment made at the beginning of the paragraph, we may assume that $v_1 = x_1$. Let $(y_1,...,y_n)$ be a witnessing sequence to $(x_1,...,x_n)$. By Lemma \ref{nochoicewitness}, $E \coloneq \langle v_1,w_1 \rangle^{\perp} \cong \langle x_1,y_1 \rangle^{\perp}$. In particular, $(v_2,...,v_k) \in \mathcal{IU}(E)$ (this is understood to be the unique length zero sequence if $k=1$), and there is a sequence of length $(n-1)$ in $\mathcal{IU}(E)$. By the inductive hypothesis applied to $E$, there exists $z \in E$ such that $(v_2,...,v_k,z) \in \mathcal{IU}(E)$. Now, $(v_1,...,v_k,z) \in \mathcal{IU}(M)$, and we are done.

    Let us now assume that $R$ is any ring satisfying\footnote{The argument we give works for any principal ideal domain.} Assumption \ref{onlyass}. We observe that the problem of extending $(v_1,...,v_k) \in \mathcal{IU}(M)$ to a sequence $(v_1,...,v_k,x) \in \mathcal{IU}(M)$ is equivalent to that of producing a unimodular vector $x$ such that $\lambda(x,x) = 0$ and $x \in e_{v,w}^{\perp}$ for some (it does not matter which, by Lemma \ref{nochoicewitness}) witnessing sequence $(w_1,...,w_k)$ to $(v_1,...,v_k)$. Let $Q(R)$ be the field of fractions of $R$. Note that $$M \otimes _{R} Q(R) \cong (e_{v,w} \otimes Q(R)) \oplus (e_{v,w}^{\perp} \otimes Q(R))$$ remains an orthogonal decomposition. We can find a non-zero isotropic vector $x' \in e_{v,w}^{\perp} \otimes Q(R)$ by our earlier efforts. Clearing denominators yields a non-zero vector $x \in e_{v,w}^{\perp}$ with $\lambda(x,x) = 0$, which we can take to be unimodular by Corollary \ref{abhayakara}.
\end{proof}

Let $z(M)$ denote the maximum length of a sequence in $\mathcal{IU}(M)$. As we have assumed the underlying $R$-module of $M$ to be finitely generated and free, this is well defined. We will refer to $z(M)$ as the isotropic rank of $M$.
\begin{prop}\label{metabolicrank}
    $M$ is metabolic if and only if $z(M) = \textup{rank}(M)/2$.
\end{prop}
\begin{proof}
    Consider some $(x_1,...,x_{z(M)}) \in \mathcal{IU}(M)$. Choosing a witnessing sequence $(y_1,...,y_{z(M)})$ to it, we obtain an embedding $e_{x,y} \subseteq M$. Thus $z(M) \leq \text{rank}(M)/2$. If $z(M) = \text{rank}(M)/2$, then $e_{x,y} = M$, so $M$ is a direct sum of forms in $\text{Met}_2(R)$ and is therefore metabolic by Proposition \ref{met=tz}. If $M$ is metabolic, it is a direct sum of forms in $\text{Met}_2(R)$ by Proposition \ref{met=tz}, and thus evidently has an isotropic unimodular sequence of length $\text{rank}(M)/2$.
\end{proof}
\begin{kor}\label{complementgenus}
    Let $(x_1,...,x_k) \in \mathcal{IU}(M)$, and let $(y_1,...,y_k)$ be a witnessing sequence to it. Then $z(e_{x,y}^{\perp}) = z(M)-k$. In particular, if $N \subseteq M$ is an embedding of metabolic forms, then $N^{\perp}$ is also metabolic.
\end{kor}
\begin{proof}
    Let $n \coloneq z(M)$. We extend $(x_1,...,x_k)$ to $(x_1,...,x_k,x'_1,...,x'_{n-k}) \in \mathcal{IU}(M)$ using Lemma \ref{sequenceext}, and pick a witnessing sequence $(z_1,...,z_k,z'_1,...,z'_{n-k})$ to it. By construction, the isotropic rank of $e_{x,z}^{\perp}$ equals $n-k$ (it has the isotropic sequence $(x'_{1},...,x'_{n-k})$ and cannot have any of longer length because concatenating with $(x_1,...,x_k)$ would produce a sequence in $\mathcal{IU}(M)$ of length $> n$). The first assertion now follows from Lemma \ref{nochoicewitness}. The second assertion follows from the first and Propositions \ref{met=tz} and \ref{metabolicrank}.
\end{proof}

We have following structural result on complements in $\mathcal{IU}(M)$, in the vein of \cite[Lem. 7.2]{unigrp} and of \cite[Lem. 3.22]{arboffset}:
\begin{prop}\label{complement}
    Let $x = (x_1,...,x_k) \in \mathcal{IU}(M)$. We have $\mathcal{IU}(M)_x \cong \mathcal{IU}(N)\langle \langle x_1,...,x_k\rangle \rangle$, for some non-degenerate symmetric bilinear form $N \subset M$ with $z(N) = z(M)-k$.
\end{prop}
\begin{proof}
    We choose a witnessing sequence $(y_1,...,y_k)$ for $(x_1,...,x_k)$, and let $N \coloneqq e_{x,y}^{\perp}$. The argument in the proof of \cite[Lem. 3.4]{charneytwo} implies that the map $\mathcal{IU}(N)\langle \langle x_1,...,x_k\rangle \rangle \to \mathcal{IU}(M)_x$ given by $((v_1,s_1),...,(v_l,s_l)) \mapsto (v_1+s_1,...,v_l+s_l)$ is an isomorphism of posets. The assertion that $z(N) = z(M)-k$ follows from Corollary \ref{complementgenus}.
\end{proof}
\begin{ddd}\label{complexity}
    We define the complexity of $M$, $c(M)$, to be the integer:
    \begin{itemize}
        \item $0$, if $2$ is invertible in $R$
        \item $\textup{dim}_{U(R)}P(M)$, otherwise
    \end{itemize}
\end{ddd}
This makes sense in light of Proposition \ref{essential}.

Recall that we denoted the hyperbolic plane $\theta(0)$ by $H$. We observe that if the isotropic rank of a form exceeds its complexity, then their difference provides a lower bound for the number of hyperbolic plane summands in the form:
\begin{lem}\label{paamarat}
    Suppose $z(M) \geq c(M)$. Then there is an embedding $H^{z(M)-c(M)} \hookrightarrow M$.
\end{lem}
\begin{proof}
    Choose $(v_1,...,v_{z(M)}) \in \mathcal{IU}(M)$ of maximal length, and let $(w_1,...,w_{z(M)})$ be a witnessing sequence. If $P(e_{v,w}) = \{0\}$, then $e_{v,w} \cong H^{z(M)}$ by Theorem \ref{myclassification}, and we are done. Else, $P(e_{v,w})$ is a non-zero $U(R)$-subspace of $P(M)$, and we can find $r_1,...,r_j \in R$ with $j = c(e_{v,w}) \leq c(M)$ such that $e_{v,w} \cong (\oplus_{i=1}^{j} \theta(r_i)) \oplus H^{z(M) - j}$, by Theorem \ref{myclassification}, and this completes the proof.
\end{proof}
We adapt a technical result from \cite{arboffset} to our setting, phrased in terms of the complexity of the form:
\begin{lem}\label{algebra}
    Suppose $M \neq 0$. Let $v_1,...,v_k$ be vectors in $M$. Suppose $k \leq z(M)-c(M) - 3$. There exists a decomposition of the form $\phi:M \cong N \oplus H^{z(M) - c(M)- k}$ such that $\phi(\langle v_1,...,v_k \rangle) \subseteq N$ and $z(N) = c(M) + k$.
\end{lem}
\begin{proof}
    We induct on $k$. Consider the case $k=1$. By Lemma \ref{paamarat}, there is a decomposition $\psi: M \cong N' \oplus H^{z(M)-c(M)}$. By Corollary \ref{abhayakara}, we can write $\psi(v_1)$ as $v'_1 + tu_1$, where $v'_1 \in N'$, $t \in R$, and $u_1 \in H^{z(M)-c(M)}$ is unimodular. As the hyperbolic plane $H$ admits a unique quadratic refinement, by \cite[Cor. 3.13]{arboffset} and \cite[Thm. 3.1]{pidusr} we may assume that $\psi(v_1)$ lives in the underlined summands of $$\underline{N' \oplus H} \oplus H^{z(M)-c(M) - 1}$$ so we can take $N$ to be $N' \oplus H$. By Corollary \ref{complementgenus}, $z(N) = c(M) +1$.

    Now suppose $k > 1$. Inductively we have a decomposition $\psi: M \cong N' \oplus H^{z(M)-c(M) - k+1}$ such that $\psi(\langle v_1,...,v_{k-1}\rangle) \subseteq N'$. As in the base case, write $\psi(v_k) = v'_k + tu_k$, where $v'_k \in N'$, $t \in R$, and $u_k \in H^{z(M)-c(M)-k+1}$ is unimodular. We again use \cite[Cor. 3.13]{arboffset} and \cite[Thm. 3.1]{pidusr} to massage $u_k$ into a single hyperbolic summand $\alpha \colon H \hookrightarrow H^{z(M)-c(M)-k+1}$, and then conclude by taking $N$ to be $N' \oplus \alpha(H)$. That $z(N) = c(M)+k$ is again checked using Corollary \ref{complementgenus}.
    \end{proof} 
For the remainder of this section, fix an isomorphism class $F \in \text{Met}_2(R)$. We let $g_F(M)$ denote the maximum integer $n$ such that $F^n$ is a direct summand of $M$.

As defined in the introduction, let $\text{Unimod}^{s}(R)^{\simeq}$ denote the groupoid of non-degenerate symmetric bilinear forms on finitely generated projective $R$-modules. This is symmetric monoidal for the direct sum. Let $\text{Met}(R)$ be the full symmetric monoidal subgroupoid spanned by the metabolic forms.

\begin{ddd}\label{cancellativeform}
    An $F$-cancellative groupoid $\mathcal{G}$ is a full subgroupoid of $\textup{Unimod}^{s}(R)^{\simeq}$ satisfying the following properties:
    \begin{enumerate}
        \item $0, F \in \mathcal{G}$
        \item If $M \in \mathcal{G}$ and $M \neq 0$, then $g_F(M) \geq 1$
        \item If $M,N \in \mathcal{G}$, then $M \oplus N \in \mathcal{G}$
        \item If $M \oplus F \in \mathcal{G}$ and $g_F(M) \geq 1$, then $M \in \mathcal{G}$
        \item If $M,N \in \mathcal{G}$ are such that $M \oplus F \cong N \oplus F$, then $M \cong N$
        \end{enumerate}
\end{ddd}
Note that Conditions 1 and 3 above imply that an $F$-cancellative groupoid is a symmetric monoidal subcategory of $\text{Unimod}^{s}(R)^{\simeq}$. We also note, by a straightforward induction argument, that Condition 4 is equivalent to the following: if $ g_F(M) \geq 1$ and $M \oplus F^n \in \mathcal{G}$ for some $n \geq 1$, then $M \in \mathcal{G}$.

For each $F \in \text{Met}_2(R)$, there is at least one $F$-cancellative groupoid: let $\text{Met}_F$ be the full subgroupoid of $\text{Met}(R)$ on $0$ and metabolic forms $M$ for which $g_F(M) \geq 1$.

\begin{theorem}\label{twotwoiscancel}
    The groupoid $\textup{Met}_F$ is $F$-cancellative.
\end{theorem}
\begin{proof}
    Conditions 1, 2 and 3 of Definition \ref{cancellativeform} are obviously satisfied by $\text{Met}_F$.

    If $M \oplus F \in \text{Met}_F$, then $M$ is metabolic by Corollary \ref{complementgenus}. Thus $M \in \text{Met}_F$ if $g_F(M) \geq 1$, verifying Condition 4 of Definition \ref{cancellativeform}.

    Suppose that $N,M \in \text{Met}_F$ satisfy $N \oplus F \cong M \oplus F$. Then $\text{rank}(N) = \text{rank}(M)$. Clearly, $N = 0$ if and only if $M=0$. Otherwise, $g_F(N) \geq 1$ and $g_F(M) \geq 1$. Thus, we have a decomposition $N \cong N' \oplus F$, so $P(N \oplus F) = P(N' \oplus F^2) = P(N') + P(F^2) = P(N') + P(F) = P(N'\oplus F) = P(N)$, by Lemma \ref{exampleparity}. Similarly, $P(M \oplus F) = P(M)$, so $P(N) = P(M)$ and Condition 5 of Definition \ref{cancellativeform} follows from Theorem \ref{myclassification}.
\end{proof}
For the rest of this section, we fix an $F$-cancellative groupoid $\mathcal{G}$ and an element $M \in \mathcal{G}$. Let $r \in R$ be such that $\theta(r) = F$. Let $\{f,f'\}$ be the standard basis of $F$ with respect to the choice $r$; we have $\lambda(f,f) = 0, \lambda(f,f') = 1, \lambda(f',f') = r$. Extending this yields a basis $\{f_1,f_1',...,f_n,f_n'\}$ of $F^n$.
    \begin{ddd}
    Let $(v_1,...,v_k) \in \mathcal{IU}(M)$. \begin{enumerate}
        \item We will say that a witnessing sequence $(w_1,...,w_k)$ to $(v_1,...,v_k)$ is partly $F$-like if the image of $\lambda(w_i,w_i)$ in $S(R)$ equals that of $r$ for each $i$, and that it is $F$-like if additionally $e_{v,w}^{\perp} \in \mathcal{G}$ and $e_{v,w}^{\perp} \neq 0$.
        \item We will say that $(v_1,...,v_k)$ is fully $F$-like if it has an $F$-like witnessing sequence.
        \item We will say that $(v_1,...,v_k)$ is parity preserving if for any witnessing sequence $(w_1,...,w_k)$ to it, we have that $P(e_{v,w}^{\perp}) = P(M)$.
    \end{enumerate}
\end{ddd}
\begin{ex}\label{tocite}
         Suppose $M = F^n \oplus T$. Consider the sequence $(f_1,...,f_n) \in \mathcal{IU}(M)$. It is clear that the sequence $(f_1',...,f_n')$ is a partly $F$-like witnessing sequence to it, and is $F$-like if and only if $T \in \mathcal{G}$ and $T \neq 0$. Moreover, we deduce from Lemma \ref{nochoicewitness} that $(f_1,...,f_n)$ is fully $F$-like if and only if $T \in \mathcal{G}$ and $T \neq 0$. Further, any strict subsequence of $(f_1,...,f_n)$ is fully $F$-like by Condition 4 of Definition \ref{cancellativeform}. 
\end{ex}
\begin{prop}\label{sankarusha}
    If $(v_1,...,v_k) \in \mathcal{IU}(M)$ is fully $F$-like, then it is parity preserving.
\end{prop}
\begin{proof}
    Let $(w_1,...,w_k)$ be a $F$-like witnessing sequence to $(v_1,...,v_k)$. Then $e_{v,w} \cong F^n$ and Condition 2 of Definition \ref{cancellativeform} yields a decomposition $e_{v,w}^{\perp} \cong F \oplus N$. Lemma \ref{exampleparity} yields $P(M) = P(e_{v,w}) + P(e_{v,w}^{\perp}) = P(F^n) + P(F) + P(N) = P(F) + P(N) = P(e_{v,w}^{\perp})$.
\end{proof}
\begin{lem}\label{createodd}
    If $(v_1,...,v_k) \in \mathcal{IU}(M)$ is parity preserving, then it has a partly $F$-like witnessing sequence.
\end{lem}
\begin{proof}
    Let $(w_1,...,w_k)$ be a witnessing sequence to $(v_1,...,v_k)$. Suppose $1 \leq i \leq k$ is such that the image of $\lambda(w_i,w_i)$ in $S(R)$ does not equal that of $r$. By assumption, there exist $x \in e_{v,w}^{\perp}$ such that the image of $\lambda(x,x)$ in $R/(2)$ equals that of $\lambda(w_i,w_i)$, and $y \in e_{v,w}^{\perp}$ such that the image of $\lambda(y,y)$ in $R/(2)$ equals that of $r$ (the latter because $g_F(M) \geq 1$ by Condition 2 of Definition \ref{cancellativeform}). Set $w'_i \coloneqq w_i + x + y$; it is clear that the image of $\lambda(w_i',w_i')$ in $R/(2)$, and therefore in $S(R)$, equals that of $r$. Then $(w_1,...,w_{i-1},w'_i,w_{i+1},...,w_k)$ remains a witnessing sequence to $(v_1,...,v_k)$. Iterating this procedure over each $j$ for which the image of $\lambda(w_j,w_j)$ in $S(R)$ does not equal $r$ yields the desired partly $F$-like witnessing sequence.
\end{proof}
We thus obtain the following characterization of fully $F$-like isotropic unimodular sequences.
\begin{kor}\label{realstuff}
    Let $(v_1,...,v_k) \in \mathcal{IU}(M)$. The following statements are equivalent:
    \begin{enumerate}
        \item $(v_1,...,v_k)$ is fully $F$-like.
        \item $(v_1,...,v_k)$ is parity preserving and for any witnessing sequence $(w_1,...,w_k)$ to it, $e_{v,w}^{\perp} \in \mathcal{G}$ and $e_{v,w}^{\perp} \neq 0$.
        \item $(v_1,...,v_k)$ is parity preserving and for any witnessing sequence $(w_1,...,w_k)$ to it, $g_F(e_{v,w}^{\perp}) \geq 1$.
        \item $(v_1,...,v_k)$ is parity preserving and for some witnessing sequence $(w_1,...,w_k)$ to it, $e_{v,w}^{\perp} \in \mathcal{G}$ and $e_{v,w}^{\perp} \neq 0$.
        \item $(v_1,...,v_k)$ is parity preserving and for some witnessing sequence $(w_1,...,w_k)$ to it, $g_F(e_{v,w}^{\perp}) \geq 1$.
    \end{enumerate}
\end{kor}
\begin{proof}
    By Lemma \ref{nochoicewitness}, $(2)$ is equivalent to $(4)$ and $(3)$ is equivalent to $(5)$. By Proposition \ref{sankarusha}, $(1)$ implies $(4)$. Condition 2 of Definition \ref{cancellativeform} shows that $(2)$ implies $(3)$. To conclude, we show that $(3)$ implies $(1)$: by Lemma \ref{createodd}, we can choose a partly $F$-like witnessing sequence $(w_1,...,w_k)$; now we have $e_{v,w}^{\perp} \oplus F^k \cong M \in \mathcal{G}$, and we conclude using Condition 4 of Definition \ref{cancellativeform}.
\end{proof}
\begin{kor}\label{makelifeeasier}
    Let $(v_1,...,v_k) \in \mathcal{IU}(M)$ be parity preserving, and $(w_1,...,w_k)$ be a witnessing sequence to it. If $e_{v,w}^{\perp}$ contains $H$ as a summand, $(v_1,...,v_k)$ is fully $F$-like.
\end{kor}
\begin{proof}
    Write $e_{v,w}^{\perp} \cong H \oplus N$. By Lemma \ref{standardparity}, $P(H) = 0$. Lemma \ref{exampleparity} implies $P(N) = P(e_{v,w}^{\perp}) = P(M)$, and we fix $s \in N$ such that the image of $\lambda(s,s)$ in $R/(2)$ equals that of $r$. Let $u_1$ be an isotropic unimodular vector in $H$, and let $u'_1 \in H$ be such that $\lambda(u_1,u_1') = 1$. Note that $\lambda(u_1',u_1') \equiv 0 (\text{mod }2)$ because $P(H) = 0$. Thus $(u'_1+s)$ is a partly $F$-like witnessing sequence to $(u_1)$ in $e_{v,w}^{\perp}$. Therefore, $g_F(e_{v,w}^{\perp}) \geq 1$, and we conclude by Corollary \ref{realstuff}.
\end{proof}
Let $\mathcal{FIU}(M)$ be the subposet of $\mathcal{IU}(M)$ consisting of the fully $F$-like sequences. This is clearly a downward closed poset. We now build up to the main result of this section which is a description of the fibres of the inclusion $\mathcal{FIU}(M) \hookrightarrow \mathcal{IU}(M)$.
\begin{lem}\label{readerfriendly}
    Let $(v_1,...,v_k) \in \mathcal{IU}(M)$. Suppose $c(M) > 0$. There exists a witnessing sequence $(w_1,...,w_k)$ to $(v_1,...,v_k)$ such that there is a direct sum decomposition $ P(M) = P(e_{v,w}) \oplus P(e_{v,w}^{\perp})$. For such a witnessing sequence $(w_1,...,w_k)$, one has that a subsequence $(v'_1,...,v'_l) \leq (v_1,...,v_k)$ is parity preserving if and only if the the image of the set $\{\lambda(w_i,w_i)|(v_i) \not \leq (v_1',...,v_l')\}$ in $R/(2)$ generates the $U(R)$-vector space $P(e_{v,w})$.
\end{lem}
\begin{proof}
We note, by Lemma \ref{nochoicewitness}, that for any witnessing sequence $(t_1,...,t_k)$, the parity space $P(e_{v,t}^{\perp})$ is independent of $(t_1,...,t_k)$. Let us denote this parity space by $S$, and choose a $U(R)$-linear splitting $P(M) = V \oplus S$. Given a witnessing sequence $(w_1',...,w'_k)$, we define $\alpha(w')$ to be the cardinality of the set of those indices $i$ for which the image of $\lambda(w'_i,w'_i)$ in $R/(2)$ lies in $V$; we will prove the first statement of the lemma by producing for any such $(w_1',...,w_k')$ a witnessing sequence $(w''_1,...,w''_k)$ with $\alpha(w'') = \alpha(w') +1$ if $\alpha(w')<k$. Indeed, if $\alpha(w') = k$, then Lemma \ref{exampleparity} yields $P(M) = P(e_{v,w'}) + S$, so $P(e_{v,w'}) = V$ and we are done.

To construct such a $(w''_1,...,w''_k)$ from a given $(w_1',...,w_k')$ with $\alpha(w') < k$, we write the image of $\lambda(w'_i,w'_i)$ in $R/(2)$ as a sum $x_i + y_i$ with $x_i \in V$ and $y_i \in S$, for each $i$. Since $\alpha(w') < k$, there exists $i$ such that $y_i \neq 0$. We choose $z_i \in e_{v,w'}^{\perp}$ such that the image of $\lambda(z_i,z_i)$ in $R/(2)$ is $y_i$. We set $w''_j \coloneq w'_j$ if $j \neq i$ and $w''_i \coloneq w_i' + z_i$, and see that $(w''_1,...,w''_k)$ is a witnessing sequence with the desired property.

We now turn to the second assertion. Let $(v_1',...,v'_l) \leq (v_1,...,v_k)$ be a subsequence and $(w'_1,...,w'_l)$ be the witnessing sequence to it obtained by taking the corresponding subsequence of $(w_1,...,w_k)$. Then $e_{v',w'}^{\perp} =  (e_{v',w'}^{\perp} \cap e_{v,w}) \oplus e_{v,w}^{\perp} $, so $M \cong e_{v',w'} \oplus (e_{v',w'}^{\perp} \cap e_{v,w}) \oplus e_{v,w}^{\perp}$. By Lemma \ref{nochoicewitness}, $(v_1',...,v'_l)$ is parity preserving if and only if $P(M) = P(e_{v',w'}^{\perp}) =P(e_{v',w'}^{\perp} \cap e_{v,w}) \oplus P(e_{v,w}^{\perp})$. This is, in turn, true if and only if $P(e_{v',w'}^{\perp} \cap e_{v,w}) = P(e_{v,w})$, which by Lemma \ref{standardparity} is equivalent to the image of the set $\{\lambda(w_i,w_i)|(v_i) \not \leq (v'_1,...,v'_l)\}$ in $R/(2)$ generating $P(e_{v,w})$ as a $U(R)$-vector space.
\end{proof}
\begin{theorem}\label{triv}
Let $(v_1,...,v_k) \in \mathcal{IU}(M)$, with $k \geq c(M) + 2$. There exists an integer $g_v$ with $0 \leq g_v \leq c(M)+1$ such that the following statements hold:
\begin{enumerate}
    \item There exists a fully $F$-like subsequence of $(v_1,...,v_k)$ of length $k-g_v$.
    \item Any fully $F$-like subsequence of $(v_1,...,v_k)$ is contained in a fully $F$-like subsequence of length $k - g_v$.
    \item A subsequence of $(v_1,...,v_k)$ of length $k - g_v$ is fully $F$-like if it is parity preserving.
\end{enumerate}
\end{theorem}
\begin{proof} 
Let $(w_1,...,w_k)$ be a witnessing sequence to $(v_1,...,v_k)$, and put $h_v \coloneq c(M) - c(e_{v,w}^{\perp})$. By Lemma \ref{nochoicewitness}, $h_v$ does not depend on the choice of witnessing sequence.

We first consider the case $h_v = 0$. Then $c(M) = c(e_{v,w}^{\perp})$, and it follows directly from the definitions that $(v_1,...,v_k)$ is parity preserving. By Example \ref{tocite}, we see that every subsequence of $(v_1,...,v_k)$ of length $k-1$ is fully $F$-like. If $e_{v,w}^{\perp} \in \mathcal{G}$ and $e_{v,w}^{\perp} \neq 0$, then $(v_1,...,v_k)$ is itself fully $F$-like. In this case, we may take $g_v = 0$. If not, we take $g_v = 1$.

Now suppose $h_v > 0$. In particular, this means that $c(M) > 0$. Using Lemma \ref{readerfriendly}, we can assume that $P(M) = P(e_{v,w}) \oplus P(e_{v,w}^{\perp})$. In particular, $h_v = c(e_{v,w})$.

Two important ramifications of this reduction are:
\begin{enumerate}
    \item  By Lemma \ref{readerfriendly}, we deduce the existence of a parity preserving sequence of length $k-h_v$ (because $k - h_v \geq k - c(M) \geq 2 > 0$), and the fact that any parity preserving subsequence of $(v_1,...,v_k)$ extends to one of length $k - h_v$.
    \item Suppose $(v_1',...,v_l')$ and $(v_1'',...,v_l'')$ are two parity preserving subsequences of $(v_1,...,v_k)$ of the same length, and $(w_1',...,w'_l)$ and $(w_1'',...,w_l'')$ are the corresponding subsequences of $(w_1,...,w_k)$. We see that $P(M) = P(e_{v',w'}^{\perp}) = P(e_{v',w'}^{\perp} \cap e_{v,w}) \oplus P(e_{v,w}^{\perp})$ and similarly $P(M) =  P(e_{v'',w''}^{\perp} \cap e_{v,w}) \oplus P(e_{v,w}^{\perp})$; therefore $P(e_{v',w'}^{\perp} \cap e_{v,w}) = P(e_{v'',w''}^{\perp} \cap e_{v,w})$. Theorem \ref{myclassification} now yields $e_{v',w'}^{\perp} = (e_{v',w'}^{\perp} \cap e_{v,w}) \oplus e_{v,w}^{\perp} \cong (e_{v'',w''}^{\perp} \cap e_{v,w}) \oplus e_{v,w}^{\perp} =e_{v'',w''}^{\perp}$.
\end{enumerate}

Let us first establish the existence of fully $F$-like subsequences of $(v_1,...,v_k)$. To do this, we consider a parity preserving subsequence $(v_1',...,v'_l)$ of length $l = k - h_v$. By Lemma \ref{createodd}, we can choose a partly $F$-like witnessing sequence $(u_1,...,u_l)$ to $(v_1',...,v'_l)$. By Example \ref{tocite}, every subsequence of $(v_1',...,v_l')$ of length $l-1$ is fully $F$-like (notice that $l - 1 > 0$).

We may therefore define $g_v$ by the condition that $k-g_v$ is the maximum length of a fully $F$-like subsequence of $(v_1,...,v_k)$; we will check presently that it satisfies the desiderata. Since fully $F$-like sequences are parity preserving (Proposition \ref{sankarusha}), Observation 2 yields (using Corollary \ref{realstuff}) that any parity preserving subsequence of length $k-g_v$ is fully $F$-like. From this and Observation 1 above, it follows that every fully $F$-like subsequence extends to one of length $k-g_v$. It remains to check that $g_v \leq c(M) + 1$. In the previous paragraph, we observed that there is a fully $F$-like subsequence of length $k- h_v -1$. It follows that $g_v \leq h_v + 1$. But $h_v = c(e_{v,w}) \leq c(M)$.
\end{proof}
We conclude this section with a brief discussion of yet another poset. Let $\mathcal{FU}(M)$ be the subposet of $\mathcal{O}(M \times M)$ consisting of sequences $((v_1,w_1),...,(v_k,w_k))$ such that $(v_1,...,v_k) \in \mathcal{IU}(M)$, $\lambda(v_i,w_j) = \lambda(w_i,w_j) = r\delta_{ij}$, and $e_{v,w}^{\perp} \in \mathcal{G}$ and $e_{v,w}^{\perp} \neq 0$. For an element $((v_1,w_1),...,(v_k,w_k)) \in \mathcal{FU}(M)$, we notice that $(w_1,...,w_k)$ is an $F$-like witnessing sequence to $(v_1,...,v_k)$ that exhibits it as fully $F$-like. There is a one-to-one correspondence between elements of $\mathcal{FU}(M)$ and embeddings $e:F^{k} \hookrightarrow M$ such that $e(F^{k})^{\perp} \in \mathcal{G}$ and $e(F^{k})^{\perp} \neq 0$, given by evaluation on the standard basis $\{f_1,...,f_k,f'_1,...,f'_k\}$.
\begin{prop}\label{complementsodd}
    Let $z = ((x_1,y_1),...,(x_k,y_k)) \in \mathcal{FU}(M)$. There is a canonical isomorphism $\mathcal{FU}(M)_z \cong \mathcal{FU}(e_{x,y}^{\perp})$.
\end{prop}
\begin{proof}
    This is immediate from the definitions and the following: for an embedding $j:F^{l} \hookrightarrow M$ that factors through $e_{x,y}^{\perp}$, we have that $(j(F^l)  \oplus e_{x,y})^{\perp} = j(F^l)^{\perp} \cap e_{x,y}^{\perp}$.
\end{proof}
\section{Connectivity of the space of destabilisations}\label{three}
In this section, we will work exclusively over a ring $R$ that satisfies Assumption \ref{onlyass}. We retain our notation and conventions from the previous section. 

 Let $M$ be a non-degenerate symmetric bilinear form and $F \in \text{Met}_2(R)$. The sequence of embeddings 
\[\begin{tikzcd}
	M & {M \oplus F} & {...} & {M \oplus F^n} & {(M \oplus F^{n}) \oplus F} & {...}
	\arrow["{\text{id} \oplus0}", hook, from=1-1, to=1-2]
	\arrow[hook, from=1-2, to=1-3]
	\arrow[hook, from=1-3, to=1-4]
	\arrow["{\text{id} \oplus0}", hook, from=1-4, to=1-5]
	\arrow[hook, from=1-5, to=1-6]
\end{tikzcd}\]
gives rise to a sequence of orthogonal groups
\[\begin{tikzcd}
	{\text{O}(M)} & {\text{O}(M \oplus F)} & {...} & {\text{O}(M \oplus F^n)} & {\text{O}(M \oplus F^{n+1})} & {...}
	\arrow[ hook, from=1-1, to=1-2]
	\arrow[hook, from=1-2, to=1-3]
	\arrow[hook, from=1-3, to=1-4]
	\arrow[ hook, from=1-4, to=1-5]
	\arrow[hook, from=1-5, to=1-6]
\end{tikzcd}\]
for which we want to establish homological stability. 

We will follow the approach outlined in Section \ref{twotwo}. To this end, let us fix an $F$-cancellative groupoid $\mathcal{G}$ for the rest of this section. As $(\mathcal{G},\oplus,0)$ is symmetric monoidal, we can form the symmetric monoidal category $(U\mathcal{G},\oplus_{\mathcal{G}},0)$.
\begin{prop}\label{localhom}
    If $M \in \mathcal{G}$, the category $U\mathcal{G}$ is locally homogeneous at $(M,F)$.
\end{prop}
\begin{proof}
    We check using Proposition \ref{crit}. The second condition is obviously true, while the first is simply Condition 5 of Definition \ref{cancellativeform}.
\end{proof}
The meat of the matter lies in establishing that the semi-simplicial sets $W_n(M,F)$ are highly connected. We will build up to this in the remainder of this section. The arguments are along similar lines to those presented in \cite{unigrp} and \cite{arboffset}. Until otherwise stated, $M$ will denote a non-degenerate symmetric bilinear form that is not necessarily an element of $\mathcal{G}$.
\begin{prop}\label{copycat}
    Let $N$ be a non-degenerate symmetric bilinear form such that $M = N$ or $M \oplus T \subseteq N$ for some $T \in \textup{Met}_2(R)$. We have
    \begin{enumerate}
        \item $\mathcal{O}(M) \cap \mathcal{U}'(N)$ is $(z(M)-3)$-connected
        \item $\mathcal{O}(M) \cap \mathcal{U}'(N)_v$ is $(z(M)-3-|v|)$-connected for all $v \in \mathcal{U}'(N)$
        \item $\mathcal{O}(M) \cap \mathcal{U}' (N) \cap \mathcal{U}(N)_{v}$ is $(z(M)-3-|v|)$-connected for all $v \in \mathcal{U}(N)$
    \end{enumerate}
    \end{prop}
    \begin{proof}
        Let $x = (x_1,...,x_{z(M)}) \in \mathcal{IU}(M)$ be of maximum length, and let $(y_1,...,y_{z(M)})$ be a witnessing sequence to it. The proposition can be proved identically to \cite[Lem. 6.8]{unigrp}, by writing $M$ as $M' \oplus e_{x,y}$ and taking $W \coloneqq \langle x_1,...,x_{z(M)}\rangle$ instead of what they consider. To illustrate the method, let $\mathcal{P}$ be the poset in item 3. Then $\mathcal{O}(W) \cap \mathcal{P} = \mathcal{O}(W) \cap \mathcal{U}(N)_{v}$ and for any $w \in \mathcal{P}$, we have $\mathcal{O}(W) \cap \mathcal{P}_w = \mathcal{O}(W) \cap \mathcal{U}(N)_{vw}$. From \cite[Lem. 5.1]{unigrp} and the fact that the stable rank of a principal ideal domain is at most $2$, it follows that $\mathcal{O}(W) \cap \mathcal{P}$ and $\mathcal{O}(W) \cap \mathcal{P}_w$ are $(z(M) - 3-|v|)-$connected and $(z(M)-3-|v|-|w|)-$connected, respectively. The conclusion now follows from \cite[Lem. 2.13(1)]{genlin}.
    \end{proof}
We now formulate the following analogue of \cite[Lem. 6.9]{unigrp}.
\begin{prop}\label{copied}
    Let $(v_1,...,v_k) \in \mathcal{U}'(M)$. The poset $\mathcal{O}(\langle v_1,...,v_k \rangle^{\perp}) \cap \mathcal{U}'(M)_{(v_1,...,v_k)}$ is $(z(M)-c(M)-3-k)$-connected.
\end{prop}
\begin{proof}
    The proof of \cite[Lem. 6.9]{unigrp} works verbatim upon using Lemma \ref{algebra} instead of \cite[Lem. 6.6]{unigrp} and Proposition \ref{copycat} instead of \cite[Lem. 6.8]{unigrp}. A sketch of the argument follows.

    Let $\mathcal{P}$ be the poset in question. There is nothing to do if $k > z(M)-c(M)-2$, so assume $k \leq z(M)-c(M)-2$. Lemma \ref{algebra} yields a decomposition $\phi:M \cong N \oplus H^{z(M)-c(M)-k}$ such that $\phi(\langle v_1,...,v_k \rangle) \subseteq N$ and $z(N) = c(M)+ k$. Hence, there exists $T \in \text{Met}_2(R)$ such that $T \subseteq N$. Let $W = \phi^{-1}(H^{z(M)-c(M)-k})$. We find that $W \oplus T \subseteq M$. Now, $\mathcal{O}(W) \cap \mathcal{P} = \mathcal{O}(W) \cap \mathcal{U}'(M)$. Let $V \coloneq \langle v_1,...,v_k \rangle$; then $N$ splits as $V \oplus P$ (as modules, not as forms!) for some $P \subset N$. Given $(w_1,...,w_l) \in \mathcal{P} - \mathcal{O}(W)$, because $(v_1,...,v_k,w_1,...,w_l) \in \mathcal{U}(M)$, it follows that $(w_1-w_1|_{V},...,w_l-w_l|_V) \in \mathcal{U}(M)$. Furthermore, $\mathcal{O}(W) \cap \mathcal{P}_{(w_1,...,w_l)} = \mathcal{O}(W) \cap \mathcal{U}'(M) \cap \mathcal{U}(M)_{(w_1-w_1|_{V},...,w_l-w_l|_V)}$. Propositions \ref{copycat} and \ref{metabolicrank}, together with  \cite[Lem. 2.13(1)]{genlin}, now imply that $\mathcal{P}$ is, in fact, $(z(M)-c(M)-3-k)$-connected.
\end{proof}
\begin{theorem}\label{alliso}
    The poset $\mathcal{IU}(M)$ is $\lfloor \frac{z(M)-c(M)-5}{2} \rfloor$-connected, and for all $y \in \mathcal{IU}(M)$, the poset $\mathcal{IU}(M)_y$ is $\lfloor \frac{z(M)-c(M)-5-|y|}{2} \rfloor$-connected.
\end{theorem}
\begin{proof}
    For the first assertion, the proof of \cite[Thm. 7.3]{unigrp} almost goes through verbatim on using Lemma \ref{algebra} instead of \cite[Lem. 6.6]{unigrp}, Proposition \ref{copied} instead of \cite[Lem. 6.9]{unigrp}, and Proposition \ref{copycat} instead of \cite[Prop. 6.8]{unigrp}. For the second assertion, the cited proof also goes through with very minor modifications: we need Lemma \ref{complement} instead of \cite[Lem. 7.2]{unigrp}.

    We include a sketch of the argument for the convenience of the reader. There is nothing to do if $z(M) -c(M) < 3$, so we assume that $z(M) - c(M) \geq 3$. For $v = (v_1,...,v_l)$ in $\mathcal{U}'(M)$, we put $X_v \coloneq \mathcal{IU}(M) \cap \mathcal{O}(\langle v_1,...,v_l \rangle^{\perp}) \cap \mathcal{U}'(M)_{(v_1,...,v_l)}$, and set \begin{equation}\label{big}
        X \coloneq \bigcup_{v \in \mathcal{U}'(M)} X_v
    \end{equation}
    Suppose $x = (x_1,...,x_k) \in \mathcal{IU}(M)$, with $k \leq z(M)-1$. Choose a witnessing sequence $(y_1,...,y_k)$ to it, and note that $e_{x,y}^{\perp}$ has an isotropic unimodular vector, by Corollary \ref{complementgenus}. Then $x \in X_v$, for each unimodular isotropic vector $v \in e_{x,y}^{\perp}$. Therefore, $\mathcal{IU}(M)_{\leq z(M)-1} = X_{\leq z(M)-1}$, where the subscripts indicate that only sequences of length at most $z(M)-1$ are under consideration. It suffices to show that $X$ is $\lfloor \frac{z(M)-c(M)-5}{2} \rfloor$-connected.

    To do this, we wish to apply Theorem \ref{nerve}. For $x = (x_1,...,x_k) \in X$, the poset $\mathcal{A}_x \coloneqq \{v \in \mathcal{U}'(M)|x \in X_v\}$ naturally identifies with $\mathcal{O}(\langle x_1,...,x_k \rangle^{\perp}) \cap \mathcal{U}'(M)_{(x_1,...,x_k)}$. The indexing poset $\mathcal{U}'(M)$ and the poset $\mathcal{A}_x$ are suitably highly connected by Propositions \ref{copycat} and \ref{copied}, respectively; it only remains to check that $X_v$ is $\text{min}\{{\lfloor \frac{z(M)-c(M)-7}{2} \rfloor},\lfloor \frac{z(M)-c(M)-3}{2}-|v| \rfloor\}$-connected and to verify the last condition in the statement of Theorem \ref{nerve}.

    We can show that $X_v$ is actually $\lfloor \frac{z(M)-c(M)-5-|v|}{2} \rfloor$-connected by descending induction on $|v|$. The base cases can be checked directly using Lemma \ref{algebra}. In general, for $w = (w_1,...,w_m) \in \mathcal{O}(\langle v_1,...,v_l \rangle ^{\perp}) \cap \mathcal{U}'(M)_{(v_1,...,v_l)}$, we set $T_w \coloneqq \mathcal{IU}(M) \cap \mathcal{O}(\langle w_1,...,w_m,v_1,...,v_l \rangle^{\perp}) \cap \mathcal{U}'(M)_{(w_1,...,w_m,v_1,...,v_l)}$, and put \begin{equation}\label{small} T \coloneqq \bigcup_{w \in \mathcal{O}(\langle v_1,...,v_l \rangle ^{\perp}) \cap \mathcal{U}'(M)_{(v_1,...,v_l)}} T_w \end{equation}
    As above, one checks that $T_{\leq z(M)-|v|-1} = {(X_v)}_{\leq z(M)-|v|-1}$, so we reduce to showing that $T$ is $\lfloor \frac{z(M)-c(M)-5-|v|}{2} \rfloor$-connected. We again wish to apply Theorem \ref{nerve} to the cover (\ref{small}), but this time, we may do so directly using the induction hypothesis and Proposition \ref{copied}.

    Let us comment on how the last condition of Theorem \ref{nerve} is checked for (\ref{small}); the case of (\ref{big}) is very similar. Consider a length one sequence $w = (w_1) \in \mathcal{O}(\langle v_1,...,v_l \rangle ^{\perp}) \cap \mathcal{U}'(M)_{(v_1,...,v_l)}$. Then $w \in \mathcal{IU}(M)$, so $w \in X_v$. Let $C_w$ be the contractible subposet of $X_v$ consisting of all sequences $(z_1,...,z_n)$ with $z_1 = w_1$. We have $T_w \subseteq T_w \cup (C_w)_{\leq z(M)-|v|-1} \subseteq T \cup (X_v)_{\leq z(M)-|v|+1} = T$. Now, the realization of $T_w \cup (C_{w}-\{w\})$ is homeomorphic to the cylinder on that of $T_w$, so $T_w \cup C_w$ is contractible. Therefore, $T_w \cup (C_w)_{\leq z(M)-|v|-1}$ is $(z(M)-|v|-1)$-connected.

    We can prove the second assertion of the theorem by induction on $|y|$, using Lemma \ref{complement} and Proposition \ref{addset}. The requirements of Proposition \ref{addset} follow from the inductive hypothesis and the first assertion of the present theorem.
\end{proof}
From this moment onward, we specialize to the case that $M \in \mathcal{G}$.
\begin{theorem}\label{orig}
    The poset $\mathcal{FIU}(M)$ is $\lfloor \frac{z(M)-3c(M)-6}{2} \rfloor$-connected.
\end{theorem}
\begin{proof}
We introduce another auxiliary poset: let $\mathcal{PIU}(M)$ be the subposet of $\mathcal{IU}(M)$ that consists of those $(v_1,...,v_k)$ that contain a fully $F$-like subsequence. By Theorem \ref{triv}, $\mathcal{PIU}(M)$ includes all sequences in $\mathcal{IU}(M)$ of length at least $c(M)+2$.

    We define a family of subposets $\{\mathcal{IU}(M)_k\}_{k \in \mathbb N}$ by setting $\mathcal{IU}(M)_k\coloneq \mathcal{PIU}(M) \cup \mathcal{IU}(M)_{\geq k+1}$. Then $\mathcal{IU}(M)_k = \mathcal{PIU}(M)$ when $k \geq c(M) +1$, and $\mathcal{IU}(M)_0= \mathcal{IU}(M)$.
    For $y \in \mathcal{IU}(M)_k - \mathcal{IU}(M)_{k+1}$, we observe that $\text{Link}_{\mathcal{IU}(M)_k}(y) = \text{Link}_{\mathcal{IU}(M)_{k+1}}(y) = \text{Link}^{+}_{\mathcal{IU}(M)_{k+1}}(y) = \text{Link}^{+}_{\mathcal{IU}(M)}(y)$. The second assertion of Theorem \ref{alliso} and Proposition \ref{the2.12thingy} applied to $d = \lfloor \frac{z(M)-c(M)-4}{2} \rfloor$ yield that $\text{Link}^{+}_{\mathcal{IU}(M)}(y)$ is $\lfloor \frac{z(M)-c(M)-6-2k}{2} \rfloor$-acyclic.

    As $\mathcal{IU}(M)_k$ is obtained by successively adding cones over  $\lfloor \frac{z(M)-c(M)-6-2k}{2} \rfloor$-acyclic links to $\mathcal{IU}(M)_{k+1}$, it follows that $H_i(\mathcal{IU}(M)_{k+1}, \mathbb Z) \to H_i(\mathcal{IU}(M)_k,\mathbb Z)$ is an isomorphism when $i \leq \lfloor \frac{z(M)-c(M)-6-2k}{2} \rfloor$. By Theorem \ref{alliso}, it follows that $\mathcal{PIU}(M)$ is $\lfloor \frac{z(M)-3c(M)-6}{2} \rfloor$-acyclic.

    Consider the inclusion $i\colon \mathcal{FIU}(M) \hookrightarrow \mathcal{PIU}(M)$. By construction, for each $y \in \mathcal{PIU}(M)$, the fibre $i/y$ is non-empty. 
    
    If $2$ is invertible in $R$, then $c(M) = 0$, and Theorem \ref{triv} implies that $i/y$ is homeomorphic to $\Delta^{ht(y)}$ (in the case $g_y = 0$ or $|y|=1$) or $S^{ht(y)-1}$ (in the case $g_y = 1$, because all subsequences are parity preserving). In particular, $i/y$ is $(ht(y)-2)$-connected. 

    Now suppose $2$ is not invertible. First, assume that $|y| \geq c(M)+2$. Writing $y = (y_1,...,y_l)$, we use Lemma \ref{readerfriendly} to choose a witnessing sequence $z = (z_1,...,z_l)$ to it with the property that $P(M) = P(e_{y,z}) \oplus P(e_{y,z}^{\perp})$. By Lemma \ref{readerfriendly}, a subsequence $y' = (y'_1,...,y'_j)$ of $y$ is parity preserving if and only if the image of the set $\{\lambda(z_i,z_i)|(y_i) \not \leq y'\}$ in the $U(R)$-vector space $P(e_{y,z})$ is a spanning set. Let $T/y$ be the poset $\{y' \in \mathcal{IU}(M)|y' \leq y, y' \text{ is parity preserving}\}$. An application of Proposition \ref{coolmatroid} shows that $T/y$ is $(ht(y) - \text{dim}_{U(R)}(P(e_{y,z}))-1)$-connected. By Theorem \ref{triv}, we have that $i/y$ is the $j$-skeleton of $T/y$, for some $j \geq ht(y)-c(M)-1$. Therefore, $i/y$ is $(ht(y)-c(M)-2)$-connected. Clearly, $i/y$ is also $(ht(y)-c(M)-2)$-connected if $|y| < c(M)+2$.

    In either case, $i/y$ is $(ht(y)-c(M)-2)$-connected. As before, from Proposition \ref{the2.12thingy} one deduces that for $y \in \mathcal{PIU}(M)$, $\text{Link}^{+}_{\mathcal{PIU}(M)}(y) = \text{Link}^{+}_{\mathcal{IU}(M)}(y)$ is $(\lfloor \frac{z(M)-c(M)-6}{2} \rfloor - ht(y))$-acyclic. By Theorem \ref{3.6thing}, it follows that $\mathcal{FIU}(M)$ is $\lfloor \frac{z(M) -3c(M)-6}{2} \rfloor$-acyclic.

    We will conclude the proof by giving a direct argument that $\mathcal{FIU}(M)$ is simply connected if $z(M)-3c(M) \geq 7$. Let $S(M)$ be the simplicial complex associated to $\mathcal{FIU}(M)$.

    Let $S(M)_{\leq 1}$ be the $1$-skeleton of $S(M)$, and pick a basepoint vertex $v = (v_1) \in S(M)_0$. By a \textit{triangle}, we will mean a loop in $S(M)_{\leq 1}$ of the form 
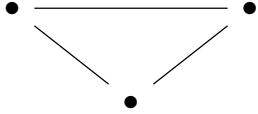
\[\begin{tikzcd}
	\bullet && \bullet \\
	& \bullet
	\arrow[no head, from=1-1, to=1-3]
	\arrow[no head, from=1-1, to=2-2]
	\arrow[no head, from=1-3, to=2-2]
\end{tikzcd}\] with a choice of clockwise or counter-clockwise orientation of the edges. We shall say that a loop $\alpha$ based at $v$ is \textit{triangular} if it is homotopic to a triangle.

We will first show that the triangular loops generate $\pi_1(S(M)_{\leq 1},v)$. Given a loop $\alpha$ based at $v$, we will show that it is homotopic to a concatenation of triangular loops by induction on $|\alpha|$, the number of edges in $\alpha$. If $|\alpha| = 3$, $\alpha$ is, by definition, a triangle. Consider the case $|\alpha| \geq 4$. We may describe $\alpha$ by a figure in $S(M)_{\leq 1}$ of the form 
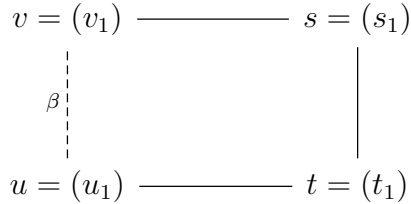
\[\begin{tikzcd}
	{v=(v_1)} && {s=(s_1)} \\
	\\
	{u = (u_1)} && {t=(t_1)}
	\arrow[no head, from=1-1, to=1-3]
	\arrow[no head, from=1-3, to=3-3]
	\arrow["\beta", dashed, no head, from=3-1, to=1-1]
	\arrow[no head, from=3-1, to=3-3]
\end{tikzcd}\] for some path $\beta$ from $u$ to $v$.
By Lemma \ref{algebra}, we obtain a decomposition of the form $M \cong N \oplus H^{z(M) -c(M)- 4}$, with $v_1,s_1,t_1,u_1 \in N$. Choose $g = (g_1) \in \mathcal{IU}(H^{z(M)-c(M)-4})$. Let $x = (x_1) \in \{v,s,t,u\}$. We claim that there is an edge $xg$ in $S(M)_{\leq 1}$. By assumption $x_1$ is unimodular, so it is also unimodular within $N$ (by Corollary \ref{subformunimod}). Therefore, we may find a witnessing sequence $y = (y_1)$ to $x$ with $y_1 \in N$. Similarly, we can find a witnessing sequence $h = (h_1)$ to $g$ with $h_1 \in H^{z(M)-c(M)-4}$, so that $H \cong e_{g,h} \subset H^{z(M) - c(M)-4}$. The sequence $(x_1,g_1) \in \mathcal{IU}(M)$ and $(y_1,h_1)$ is a witnessing sequence to it; moreover, we have $P(M) = P(e_{x,y}^{\perp}) = P(\langle x_1,y_1,g_1,h_1\rangle^{\perp} \oplus H)=  P(\langle x_1,y_1,g_1,h_1\rangle^{\perp})$ by Lemma \ref{exampleparity} (recall that $P(H) = 0$ by Lemma \ref{standardparity}). Thus, $(x_1,g_1)$ is parity preserving. It is also clear that $\langle x_1,y_1,g_1,h_1\rangle^{\perp}$ has $H$ as a direct summand. Corollary \ref{makelifeeasier} now yields that $xg = (x_1,g_1) \in \mathcal{FIU}(M)$. Therefore, we obtain
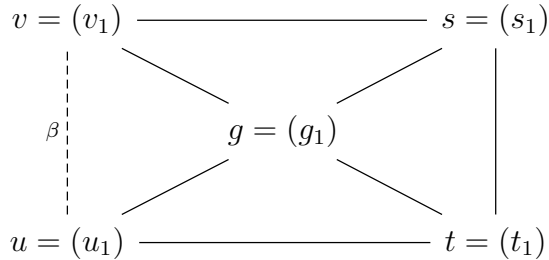
\[\begin{tikzcd}
	{v=(v_1)} && {s=(s_1)} \\
	& {g=(g_1)} \\
	{u = (u_1)} && {t=(t_1)}
	\arrow[no head, from=1-1, to=1-3]
	\arrow[no head, from=1-1, to=2-2]
	\arrow[no head, from=1-3, to=3-3]
	\arrow[no head, from=2-2, to=1-3]
	\arrow[no head, from=2-2, to=3-1]
	\arrow[no head, from=2-2, to=3-3]
	\arrow["\beta", dashed, no head, from=3-1, to=1-1]
	\arrow[no head, from=3-1, to=3-3]
\end{tikzcd}\]
We now observe that $\alpha$ is homotopic to $$(vsgv)(vgstgv)(vgtugv)(vgu)(\beta)$$ which is the concatenation of three triangular loops and the loop $(vgu)\beta$ that has one less edge than $\alpha$, so we are done by induction.

Finally, we will show that each triangle in $S(M)_{\leq 1}$ becomes trivial in $S(M)_{\leq 2}$. This will complete the proof because $\pi_1(S(M)_{\leq 1},v) \to \pi_1(S(M),v)$ is surjective.

So consider a triangle 
\[\begin{tikzcd}
	& {s=(s_1)} \\
	{t=(t_1)} && {u=(u_1)}
	\arrow[no head, from=1-2, to=2-1]
	\arrow[no head, from=1-2, to=2-3]
	\arrow[no head, from=2-1, to=2-3]
\end{tikzcd}\]
We will use a more refined version of the argument above. By Lemma \ref{algebra}, we can find a decomposition $M \cong N \oplus H^{z(M)-c(M)-3}$ such that $s_1,t_1,u_1 \in N$. As before, we pick $g = (g_1) \in \mathcal{IU}(H^{z(M) -c(M)- 3})$ and a witnessing sequence $h = (h_1)$ to it with $h_1 \in H^{z(M) -c(M)- 3}$. Now, let $x = (x_1),y = (y_1)\in \{s,t,u\}$ be distinct. By assumption, the sequence $(x_1,y_1)$ is unimodular, so it is unimodular inside $N$. Therefore we may choose a witnessing sequence $(w_1,z_1)$ to it that also lives inside $N$. We note that $(x_1,y_1,g_1) \in \mathcal{IU}(M)$ and that $(w_1,z_1,h_1)$ is a witnessing sequence to it. Exactly as before, we argue that $(x_1,y_1,g_1)$ is parity preserving and $\langle x_1,y_1,g_1,w_1,z_1,h_1 \rangle^{\perp}$ contains $H$ as a direct summand, and then use Corollary \ref{makelifeeasier} to deduce that $(x_1,y_1,g_1) \in \mathcal{FIU}(M)$. We thus obtain the skeleton of a tetrahedral shape:
\[\begin{tikzcd}
	& {g=(g_1)} \\
	\\
	& {s=(s_1)} \\
	{t=(t_1)} && {u=(u_1)}
	\arrow[no head, from=1-2, to=3-2]
	\arrow[no head, from=1-2, to=4-1]
	\arrow[no head, from=1-2, to=4-3]
	\arrow[no head, from=3-2, to=4-1]
	\arrow[no head, from=3-2, to=4-3]
	\arrow[no head, from=4-1, to=4-3]
\end{tikzcd}\]
Except for the base triangle $stu$, all the other faces of the tetrahedron are filled up in $S(M)_{\leq 2}$ because they correspond to length three sequences in $\mathcal{FIU}(M)$. This, however, implies that the triangle $stu$ also becomes trivial in $S(M)_{\leq 2}$, finishing the proof.
\end{proof}
\begin{theorem}\label{posetconn}
    The poset $\mathcal{FU}(M)$ is $\lfloor \frac {z(M)-3c(M)-7}{2} \rfloor$-connected, and for all $z \in \mathcal{FU}(M)$, the poset $\mathcal{FU}(M)_z$ is $\lfloor \frac {z(M)-3c(M)-7-|z|}{2} \rfloor$-connected.
\end{theorem}
\begin{proof}
    We replicate the approach used in the proof of \cite[Thm. 7.4]{unigrp}. The crucial tool is the nerve theorem, Theorem \ref{nerve}. We will proceed by induction on $z(M)-3c(M)$. Furthermore, it suffices to only prove the first assertion by Proposition \ref{complementsodd} and Corollary \ref{complementgenus}. The cases when $z(M)-3c(M) \leq 4$ are trivial. 
    
    For $v = (v_1,...,v_k) \in \mathcal{FIU}(M)$, let $X_v \subset \mathcal{FU}(M)$ consist of those sequences $z =((x_1,y_1),...,(x_l,y_l))$ such that $\langle v_1,...,v_k \rangle \subset e_{x,y}^{\perp}$ and $(x_1,...,x_l,v_1,...,v_k) \in \mathcal{FIU}(M)$. This definition in particular asserts that $(v_1,...,v_k) \in \mathcal{FIU}(e_{x,y}^{\perp})$ if $z \in X_v$, so we may equivalently define $X_v$ to consist of those sequences $z = ((x_1,y_1),...,(x_l,y_l))$ for which there exists a witnessing sequence $(w_1,...,w_k)$ to $(v_1,...,v_k)$ such that $((v_1,w_1),...,(v_k,w_k)) \in \mathcal{FU}(M)_{z}$. Observe that if $v \leq v'$, then $X_{v'} \subseteq X_v$. We put $$X \coloneq \bigcup _{v \in \mathcal{FIU}(M)} X_v$$
We claim that $X_{\leq z(M) -3c(M)-2} = \mathcal{FU}(M)_{\leq z(M)-3c(M)-2}$. To verify this, consider an element $((x_1,y_1),...,(x_l,y_l)) \in \mathcal{FU}(M)$ with $l \leq z(M)-3c(M)-2$. In light of Proposition \ref{complementsodd}, we need to show that $\mathcal{FU}(e_{x,y}^{\perp})$ is non-empty. By Corollary \ref{complementgenus} and Lemma \ref{paamarat}, we can select an embedding $H^2 \hookrightarrow e_{x,y}^{\perp}$. Let $(z_1,z_2) \in \mathcal{IU}(e_{x,y}^{\perp})$ and let $(z_1',z_2')$ be a witnessing sequence to it; we have $\langle z_i,z_1'\rangle \cong \langle z_2,z_2'\rangle \cong H$. It is clear that $(z_1) \in \mathcal{IU}(e_{x,y}^{\perp})$ is parity preserving, and $\langle z_1,z_1'\rangle^{\perp} \cap e_{x,y}^{\perp}$ contains $H$ as a direct summand. The claim now follows from Corollary \ref{makelifeeasier}. 
    
    It suffices to show that $X$ is $\lfloor \frac {z(M)-3c(M)-7}{2} \rfloor$-connected. We first remark that the indexing poset $\mathcal{FIU}(M)$ is $\lfloor \frac {z(M)-3c(M)-6}{2} \rfloor$-connected by Theorem \ref{orig}. For $z \in X$, set $A_z\coloneqq \{v \in \mathcal{FIU}(M)| z \in X_v\}$. Writing $z = ((x_1,y_1),...,(x_l,y_l))$, it is immediate from the definitions and Proposition \ref{complementsodd} that there is a canonical isomorphism $A_z \cong \mathcal{FIU}(e_{x,y}^{\perp})$. Further, $z(e_{x,y}^{\perp}) = z(M) -|z|$ by Corollary \ref{complementgenus} and $c(e_{x,y}^{\perp}) = c(M)$ because $z$ is fully $F$-like. Thus, Theorem \ref{orig} yields that $A_z$ is $\lfloor \frac {z(M)-3c(M)-|z|-6}{2} \rfloor$-connected. In particular, it is $(\lfloor \frac {z(M)-3c(M)-7}{2} \rfloor - |z| + 1)$-connected.

    Let $v = (v_1,...,v_k) \in \mathcal{FIU}(M)$, and let $(w_1,...,w_k)$ be an $F$-like witnessing sequence to it. We now observe that projecting onto $e_{v,w}^{\perp}$ defines a map $\phi: X_v \to \mathcal{FU}(e_{v,w}^{\perp})$. To see this, consider an element $((x_1,y_1),...,(x_l,y_l)) \in X_v$, and for each $i$, write $x_i = x_i' + s_i$, $y_i = y_i' + t_i$ where $x_i',y_i' \in e_{v,w}^{\perp}$ and $s_i,t_i \in e_{v,w}$. Any element $p \in \{s_i,t_i|i \in \{1,...,l\}\}$ may be written as a linear combination $p = \Sigma_{j=1}^{k}(a_jv_j + b_jw_j)$, and the condition that $\lambda(x_i,v_j) = \lambda(y_i,v_j) = 0$ for all $i$ and $j$ forces 
    $b_j = 0$ for all $j$. Thus, each $s_i$ and $t_i$ is a linear combination of $v_1,...,v_k$; from this it follows that $\lambda(x'_i,x'_j) = 0, \lambda (x_i',y_j') = 1$, and $\lambda (y'_i,y'_j) = \lambda(y_i,y_j)\delta_{ij}$. It remains to verify that the partly $F$-like witnessing sequence $(y_1',...,y_l')$ is $F$-like in $e_{v,w}^{\perp}$. To do this, one has to check that $e_{x',y'}^{\perp} \cap e_{v,w}^{\perp} \in \mathcal{G}$ and $e_{x',y'}^{\perp} \cap e_{v,w}^{\perp} \neq 0$. The subform  $e_{x',y'} \oplus e_{v,w}$ is non-degenerate, being a direct sum of copies of $F$, so we have $M \cong (e_{x',y'}^{\perp} \cap e_{v,w}^{\perp}) \oplus e_{x',y'} \oplus e_{v,w}$. By assumption, $(x_1,...,x_l,v_1,...,v_k) \in  \mathcal{IU}(M)$, and its span is contained in $e_{x',y'} \oplus e_{v,w}$. Therefore, one can find a witnessing sequence $(u_1,...,u_{k+l})$ for it that lives in $e_{x',y'} \oplus e_{v,w}$. Counting dimensions, we see that the orthogonal complement $\langle x_1,...,x_l,v_1,...,v_k,u_1,...,u_{k+l} \rangle ^{\perp} = e_{x',y'}^{\perp} \cap e_{v,w}^{\perp}$. Because $(x_1,...,x_l,v_1,...,v_k)$ is fully $F$-like, Corollary \ref{realstuff} implies that $\langle x_1,...,x_l,v_1,...,v_k,u_1,...,u_{k+l} \rangle^{\perp} \in \mathcal{G}$ and is non-zero, and this completes the verification.

     The map $\phi: X_v \to \mathcal{FU}(e^{\perp}_{v,w})$ refines to a map $\psi: X_v \to \mathcal{FU}(e_{v,w}^{\perp})\langle \langle v_1,...,v_k \rangle \times \langle v_1,...,v_k \rangle \rangle$ upon additionally preserving the data of the $s_i$ and $t_i$ as defined in the previous paragraph. We claim that $\psi$ is an isomorphism. It is obviously injective.
    
    To show surjectivity, consider an arbitrary element $(((x_1',y_1'),(s_1,t_1)),...,((x'_l,y'_l),(s_l,t_l)))$ in the target. Set $x_i \coloneqq x'_i + s_i$ and $y_i \coloneqq y'_i + t_i$; we need to show that $z\coloneqq ((x_1,y_1),...,(x_l,y_l)) \in X_v$. It is clear that $\lambda(x_i,x_j) = 0$, $\lambda(x_i,y_j) = \lambda(y_i,y_j) = \delta_{ij}$, and $\langle v_1,...,v_k \rangle \subseteq e_{x,y}^{\perp}$. Because $((x_1',y_1'),...,(x_k',y_k')) \in \mathcal{FU}(e_{v,w}^{\perp})$, it follows that $e_{x',y'}^{\perp} \cap e_{v,w}^{\perp} \in \mathcal{G}$ and is non-zero. The sequence $(x_1,...,x_l,v_1,...,v_k)$ is unimodular because $(x_1',...,x_l',v_1,...,v_k)$ is, and moreover, it is unimodular inside $e_{x',y'} \oplus e_{v,w}$ for the same reason. As in the argument for well-definedness of $\phi$, we can choose a witnessing sequence $(u_1,...,u_{l+k})$ for $(x_1,...,x_l,v_1,...,v_k)$ that lives in $e_{x',y'} \oplus e_{v,w}$. Now by a dimension count, we see that the orthogonal complement $\langle x_1,...,x_l,v_1,...,v_k,u_1,...,u_{l+k} \rangle ^{\perp} = e_{x',y'}^{\perp} \cap e_{v,w}^{\perp}$, which lies in $\mathcal{G}$ and is non-zero. Therefore, $(x_1,...,x_l,v_1,...,v_k) \in \mathcal{FIU}(M)$ as required.

    We have $z(e_{v,w}^{\perp}) = z(M) - |v|$ by Corollary \ref{complementgenus}, and $c(M) = c(e_{v,w}^{\perp})$ because $(v_1,...,v_k)$ is fully $F$-like. By induction, $\mathcal{FU}(e_{v,w}^{\perp})$ is $\lfloor \frac {z(M)-3c(M)-7-|v|}{2} \rfloor$-connected; in particular, it is $\text{min}\{1,\lfloor \frac{z(M)-3c(M)-7-|v|}{2}\rfloor-1\}$-connected. The inductive hypothesis also implies that for $z' \in \mathcal{FU}(e_{v,w}^{\perp})$, the complement $\mathcal{FU}(e_{v,w}^{\perp})_{z'}$ is $\lfloor \frac {z(M)-3c(M)-7-|v|-|z'|}{2} \rfloor$-connected. The conditions of Proposition \ref{addset} are thereby satisfied, and it follows that $X_v$ is $\lfloor \frac {z(M)-3c(M)-7-|v|}{2} \rfloor$-connected. In particular, it is $\text{min}\{\lfloor \frac {z(M)-3c(M)-7}{2} \rfloor - 1, \lfloor \frac {z(M)-3c(M)-7}{2} \rfloor - |v| + 1\}$-connected.

    It remains to construct, for each $v = (v_1) \in \mathcal{FIU}(M)$ with $|v| = 1$, a $\lfloor \frac {z(M)-3c(M)-7}{2} \rfloor$-connected $Y_v$ such that $X_v \subseteq Y_v \subseteq X$. This can be done in an identical manner to the proof of \cite[Thm. 7.4]{unigrp}. We detail this argument below for the convenience of the reader. We choose an $F$-like witnessing sequence $w = (w_1)$ to $v$. Let $C_{v,w}$ be the contractible subposet of $\mathcal{FU}(M)$ consisting of those sequences $((x_1,y_1),...,(x_l,y_l))$ in $X$ for which $(x_1,y_1) = (v_1,w_1)$. Let $D_{v,w} \coloneq \mathcal{FU}(e_{v,w}^{\perp})$, which canonically identifies with a subposet of $X_v$ by Proposition \ref{complementsodd}. Note that $z(e_{v,w}^{\perp}) = z(M)-1$ and $c(e_{v,w}^{\perp}) = c(M)$ as in the previous paragraph. By induction, $D_{v,w}$ is $\lfloor \frac {z(M)-3c(M)-8}{2} \rfloor$-connected. Put $Y'_v \coloneq (C_{v,w})_{\leq z(M)-3c(M)-2} \cup (D_{v,w})_{\leq z(M)-3c(M)-2}$, and $Y_v \coloneq Y'_v \cup X_v$. As the realization of the poset $(C_{v,w}-\{(v,w)\}) \cup D_{v,w}$ is homeomorphic to the cylinder on $D_{v,w}$, it follows that $C_{v,w} \cup D_{v,w}$ is contractible. Therefore, $Y'_v$ is $(z(M)-3c(M)-3)$-connected. Set $n \coloneq \lfloor \frac {z(M)-3c(M)-7}{2} \rfloor$. A Mayer-Vietoris argument yields that $Y_v$ is $(n-1)$-acyclic, and an exact sequence
    
\begin{tikzcd}
	{H_n(D_{v,w},\mathbb Z)} & {H_n(X_v,\mathbb Z)} & {H_n(Y_v,\mathbb Z)} & 0
	\arrow["\beta", from=1-1, to=1-2]
	\arrow[from=1-2, to=1-3]
	\arrow[from=1-3, to=1-4]
\end{tikzcd}

    Under the isomorphism $\psi: X_v \cong D_{v,w}\langle \langle v_1\rangle \rangle$ established above, $\alpha$ is induced by the map $l_0\colon D_{v,w} \to D_{v,w} \langle \langle v_1 \rangle \rangle$. The hypothesis of Proposition \ref{addset} is satisfied by induction, so it follows that $\beta$ is an isomorphism and hence that $Y_v$ is $n$-acyclic. When $n \geq 1$, the van Kampen theorem yields $\pi_1(Y_v) \cong \pi_1(X_v)/N$, where $N$ is the normal subgroup generated by the image of $l_0$. As $D_{v,w}$ is $(n-1)$-connected, Proposition \ref{addset} again kicks in to show that $\pi_1(Y_v) = 0$. Thus, $Y_v$ is $n$-connected and $X_v \subseteq Y_v \subseteq X$.
    
    The theorem now follows from Theorem \ref{nerve}. 
\end{proof}
We can now deduce the main result of this paper:
\begin{theorem}\label{mainl}
    Let $\mathcal{A}$ be a coefficient system on $\mathcal{G}$ of degree $r$ at $0$. Then the map induced by $- \oplus F$ on homology with coefficients in $\mathcal{A}$ $$H_i(\textup{O}(M),\mathcal{A}(M)) \to H_i(\textup{O}(M \oplus F),\mathcal{A}(M \oplus F))$$ is an epimorphism if $i \leq \frac{z(M)- 3c(M)- 5}{2} - r$ and an isomorphism if $i \leq \frac {z(M)-3c(M)- 7}{2} - r$. If $\mathcal{A}$ is split, then it is an epimorphism if $i \leq \frac{z(M) -3c(M) - 5-r}{2}$ and an isomorphism if $i \leq \frac {z(M)-3c(M) - 7-r}{2}$. If $\mathcal{A}$ is constant, then it is an epimorphism if $i \leq \frac{z(M) -3c(M) - 5}{2}$ and an isomorphism if $i \leq \frac{z(M) -3c(M)- 6}{2}$.
\end{theorem}
\begin{proof}
   Recall from Proposition \ref{localhom} that $U\mathcal{G}$ is locally homogeneous at $(M,F)$. Furthermore, Proposition \ref{ff} shows that the automorphism groups in $U\mathcal{G}$ are the same as those in $\mathcal{G}$. The face poset of the space of destabilisations $W_n(M,F)$ is canonically isomorphic to $\mathcal{FU}(M \oplus F^n)$, which is $\lfloor \frac{z(M)-3c(M)+n-7}{2} \rfloor$-connected by Theorem \ref{posetconn}. In particular, it is $\frac{n-2}{2}$-connected when $z(M)-3c(M) \geq 5$. The theorem now follows by Theorem \ref{addlater} (note that it is vacuously true when $z(M)-3c(M) < 5$).
\end{proof}
In particular, we recover Theorem \ref{mainlintro} by taking $\mathcal{G} = \text{Met}_F$ in light of Theorem \ref{twotwoiscancel}. We also remark that a version of Theorem \ref{mainl} for abelian coefficients (see Remark \ref{abeliancoefficient}) can be deduced from Theorem \ref{posetconn} using \cite[Thm. 3.4]{machine} and \cite[Thm. 4.20.]{machine}.
    It is conceivable that the slope in the stable range can be improved from $1/2$ to $2/3$ using the methods of \cite{cellular}, in line with other examples.
\begin{rem}
    In \cite[Sec. 7.3]{krannichimpr}, Krannich explains that a homological stability result akin to the above can be deduced without any conditions on the underlying symmetric monoidal groupoid (like the ones that we have imposed in Definition \ref{cancellativeform}) if a certain subsemisimplicial set of $W_n(M,F)$ (corresponding to maps $F^{p+1} \to M \oplus F^n$ whose complement is isomorphic to $M \oplus F^{n-p-1}$) is highly connected. However, in our situation, we do not know how to prove that this is the case in general, and can only prove high connectivity when the complements are allowed to vary in the $F$-cancellative groupoid $\mathcal{G}$.
\end{rem}
We conclude this section by noting that all forms in $\text{Met}_2(R)$ induce the same isomorphism in homology in the stable range. Let us restrict ourselves to constant coefficients for the sake of brevity. We remind the reader that $M$ is assumed to be an object of some $F$-cancellative groupoid $\mathcal{G}$, for a fixed $F \in \text{Met}_2(R)$. The precise statement is:
\begin{kor}\label{allsame}
    Let $F' \in \textup{Met}_2(R)$, and let $\mathcal{G}'$ be a $F'$-cancellative groupoid. Suppose $M \in \mathcal{G}'$ and $M \neq 0$. Then there exists an isomorphism $\alpha: M \oplus F \xrightarrow{\cong} M \oplus F'$ such that the following triangle commutes 
\[\begin{tikzcd}
	& {H_i(O(M),A)} \\
	{H_i(O(M \oplus F),A)} && {H_i(O(M \oplus F'),A)}
	\arrow[from=1-2, to=2-1]
	\arrow[from=1-2, to=2-3]
	\arrow["{\alpha_*}"', from=2-1, to=2-3]
\end{tikzcd}\]
for constant coefficients $A$, for $i \leq \frac{z(M) -3c(M)- 6}{2}$.
\end{kor}
\begin{proof}
     Let us assume that $z(M)-3c(M) \geq 6$; otherwise, there is nothing to do. Condition 2 of Definition \ref{cancellativeform} implies that $g_F(M) \geq 1$, so we have a splitting $M \cong N' \oplus F$. By Corollary \ref{complementgenus}, there is a sequence $(z_1,...,z_{z(M)-1}) \in \mathcal{IU}(N')$. Let $(z_1',...,z_{z(M)-1}')$ be a witnessing sequence to it in $N'$. Since $z(N') = z(M)-1 \geq c(M) \geq c(N')$, Lemma \ref{paamarat} implies that $H$ is a direct summand of $N'$. We may therefore write $N' = N'' \oplus H$. Put $N \coloneq N'' \oplus F$. Then $M \cong N \oplus F$ (using Lemma \ref{addition}) and $c(N) = c(M)$ (using Lemma \ref{exampleparity} and that $g_F(N) \geq 1$ by construction). Fix some isomorphism $\psi \colon F^2 \cong F \oplus H$. Consider the following diagram of forms:
     
\begin{tikzcd}
	&& N \\
	&& {N \oplus F} \\
	{N \oplus F^2} &&&& {N \oplus F \oplus H}
	\arrow[from=1-3, to=2-3]
	\arrow[from=1-3, to=3-1]
	\arrow[from=1-3, to=3-5]
	\arrow[from=2-3, to=3-1]
	\arrow[from=2-3, to=3-5]
	\arrow["{N \oplus \psi}"', from=3-1, to=3-5]
\end{tikzcd}

All triangles except the inner middle one commute. Taking homology, we observe that $H_i(O(N),A) \to H_i(O(N \oplus F),A)$ is an epimorphism by Theorem \ref{mainl}; therefore, the following triangle commutes:

\[\begin{tikzcd}
	&& {H_i(O(M),A)} \\
	{H_i(O(M \oplus F),A)} &&&& {H_i(O(M \oplus H),A)}
	\arrow[from=1-3, to=2-1]
	\arrow[from=1-3, to=2-5]
	\arrow["{\beta_*}"', from=2-1, to=2-5]
\end{tikzcd}\]
Here, $\beta \colon M \oplus F \xrightarrow{\cong} M \oplus H$ is the composite $M \oplus F \cong N \oplus F^2 \xrightarrow{N \oplus \psi} N \oplus F \oplus H \cong M \oplus H$.

We can repeat this argument with $F$ replaced by $F'$ to obtain a suitable isomorphism $\beta': M \oplus F' \xrightarrow{\cong} M \oplus H$, and then take $\alpha \coloneq (\beta')^{-1} \circ\beta$.
\end{proof}
Specialising to the $F$-cancellative groupoid $\text{Met}_F$ (Theorem \ref{twotwoiscancel}), we recover Corollary \ref{sameisointro} from the introduction.

There seems to be no obvious reason for the result of Corollary \ref{allsame} to hold outside the stable range. With $M$ as in the statement of the corollary, we notice that every element in the image of the map $O(M) \hookrightarrow O(M \oplus F)$ fixes a particular embedding of $F$ in $M \oplus F$; the same is true for any conjugate of this subgroup. In particular, if there is no non-trivial vector in $M$ that is fixed by every element in $O(M)$, it follows that for $F' \neq F$, the stabilisation maps $O(M) \hookrightarrow O(M \oplus F)$ and $O(M) \hookrightarrow O(M \oplus F')$ are not conjugate for any choice of isomorphism $M \oplus F \cong M \oplus F'$.
\section{Applications}
Let $R$ be a ring satisfying Assumption \ref{onlyass}. We recall from the introduction that the symmetric Grothendieck-Witt theory of $R$ can be defined as the group completion $$\mathcal{GW}^{s}(R) \coloneq (\text{Unimod}^{s}(R)^{\simeq},\oplus)^{\text{grp}}$$
 We also recall that a form $M \in \text{Unimod}^{s}(R)^{\simeq}$ is said to be \textit{cofinal} if for any $N \in \text{Unimod}^{s}(R)^{\simeq}$, there exists $n \in \mathbb Z_{>0}$ and $N' \in \text{Unimod}^{s}(R)^{\simeq}$ such that $N \oplus N' \cong M^{n}$. As explained in the introduction, the stable (co)homology of orthogonal groups of cofinal forms can be accessed through Grothendieck-Witt theory by means of the group completion theorem.
 
 We begin this section by establishing the criterion for the existence of cofinal forms stated in the introduction (Theorem \ref{cofincritintro}).
\begin{theorem}\label{cofincrit}
    The following statements are equivalent:
    \begin{enumerate}
        \item There exists a cofinal form in $\textup{Unimod}^{s}(R)^{\simeq}$ that is metabolic.
        \item There exists a cofinal form in $\textup{Unimod}^{s}(R)^{\simeq}$.
        \item There exists a form $M$ in $\textup{Unimod}^{s}(R)^{\simeq}$ such that $P(M) = R/(2)$.
        \item If $2$ is not invertible in $R$, $R/(2)$ is finite dimensional over $U(R)$.
    \end{enumerate}
    In this case, a metabolic form is cofinal if and only if its parity equals $R/(2)$.
\end{theorem}
\begin{proof}
$\underline{(1) \implies (2)}$: Obvious.

$\underline{(2) \implies (3)}$: If $M$ is cofinal, then for each $r \in  R/(2)$, there exists $n$ such that $\theta(r) \hookrightarrow M^n$. Thus, $r \in P(\theta(r)) \subset P(M^n) = P(M)$ by Lemma \ref{exampleparity}, and $P(M) = R/(2)$.

$\underline{(3) \implies (4)}$: This follows from Proposition \ref{essential}.

$\underline{(4) \implies (1)}$: As noted in the introduction, metabolic forms are cofinal in all forms: given $N \in \text{Unimod}^{s}(R)^{\simeq}$, we let $-N$ be the form with the same underlying $R$-module as $N$, and $\lambda_{-N}(x,y) = -\lambda_N(x,y)$ for all $x,y$; then $N \oplus (-N)$ is metabolic with the diagonal $N \to N \oplus(-N)$ serving a Lagrangian. Therefore, it suffices to exhibit a metabolic form that is cofinal among all metabolic forms. Since metabolic forms are closed under direct sums and classified by their rank and parity (Theorem \ref{myclassification}), it follows that if a metabolic form $M$ has parity $R/(2)$, then it is cofinal. We can build such an $M$ as follows - if $2$ is invertible, we can take $M \coloneq H$; else we choose a $U(R)$-spanning set $r_1,...,r_n$ of $R/(2)$ and put $M \coloneq \oplus_{i=1}^{n}\theta(r_i)$ (Lemma \ref{standardparity}).

The remaining assertion in the statement of the theorem follows immediately from an examination of the arguments above.
\end{proof}
An immediate application of this result is that a field of characteristic $2$ admits a cofinal form if and only if it is finite over its Frobenius. We proceed to discuss more examples below.
\begin{ex}\label{cofinexamb} 
    \begin{enumerate}
        \item The form $\theta(1)$ is cofinal over $\mathbb Z$ as $U(\mathbb Z) = \mathbb Z/(2)$. This may also be deduced from the classical fact that non-degenerate symmetric bilinear forms over the integers are classified by their rank, signature, and parity (see \cite[Chp. 5, Sec. 2.2, Thm. 6]{arithmetic}). 
        \item The rank one form $[1]$ is cofinal over any perfect field $F$ of characteristic $2$. Indeed, because every element is a square, the set $S(F)$ has two elements and thus the form $\theta(1)$ is cofinal. Now, $\theta(1) \cong \text{diag}(1,-1) \cong \text{diag}(1,1)$, whence the statement follows.
        \item The form $\text{diag}(1,i)$ is cofinal over $\mathbb Z[i]$. To see this, observe that $U(\mathbb Z[i]) = \mathbb F_2$ and $\mathbb Z[i]/(2) \cong \mathbb F_2[x]/(x^2)$, so $\theta(1) \oplus \theta(i)$ is cofinal. Further, $\theta(1) \cong \text{diag}(1,1)$ and $\theta(i) \cong \text{diag}(i,i)$, and the statement follows.
        \item The form $\theta(1) \cong \text{diag}(1,-1)$ is cofinal over $\mathbb Z[\omega]$. This is because $U(\mathbb Z[\omega]) = \mathbb Z[\omega]/(2)$, as we explained in the course of proving Proposition \ref{whichnumberrings}.
    \end{enumerate}
\end{ex}
Of independent interest is that all number rings, not just the ones satisfying Assumption \ref{onlyass}, admit cofinal forms:
\begin{prop}\label{allnumbring}
    For any number ring $S$, there is a cofinal form in $\textup{Unimod}^{s}(S)^{\simeq}$.
\end{prop}
\begin{proof}
     It is enough to find a metabolic form that is cofinal among all metabolic forms. The ring $S/(2)$ is finite because $S$ is a lattice over $\mathbb Z$; therefore the form $\bigoplus_{s \in S/(2)} \theta(s)$ is cofinal among all metabolic forms with free Lagrangian, as follows from Propositions \ref{thetasurj} and \ref{met=tz}. It now suffices to show that such metabolic forms are cofinal in all metabolic forms. Let $M$ be a metabolic form with Lagrangian $L$; because $S$ is a Dedekind domain, we find that $L$ is isomorphic to $S^n \oplus \mathcal{L}$ for some line bundle $\mathcal{L}$. As the Picard group $\text{Pic}(S)$ is finite, we can set $t$ to be the order of $\mathcal{L} \in \text{Pic}(S)$. The form $M^{t}$ then admits as Lagrangian the module $S^{nt} \oplus \mathcal{L}^{t} \cong S^{nt + t -1} \oplus \mathcal{L}^{\otimes t} \cong S^{(n+1)t}$, which is free. This concludes the proof.
\end{proof}

For the remainder of this section, we will assume that $R$ satisfies the conditions of Theorem \ref{cofincrit}. In this case, let $c(R)$ be $0$ or $\text{dim}_{U(R)}(R/(2))$ depending on whether $2$ is invertible in $R$ or not. Note that $c(R) = c(M)$ for any cofinal form $M$. We restrict ourselves to constant coefficients in our homological stability statements for the sake of brevity.
\begin{theorem}\label{homstabcofin}
    Let $M \in \textup{Unimod}^{s}(R)^{\simeq}$ be the cofinal form of minimal rank that is metabolic. For constant coefficients ${A}$, the map $$H_i(O(M^n),{A}) \to H_i(O(M^{n+1}),{A})$$ is an isomorphism if $i \leq \frac{(n-3)c(R)-6}{2}$ and $c(R) \neq 0$, or if $i \leq \frac{{n}-6}{2}$ and $c(R) = 0$.
\end{theorem}
\begin{proof}
    In the case $c(R) = 0$, we see that $M = H$, and the result follows immediately from Theorem \ref{mainl}. The offset in our bound is slightly worse than that of \cite[Thm. 5.15]{machine}.

    Suppose $c(R) \neq 0$. We observe that $z(M) = c(M) = c(R)$, with the first equality being a consequence of Theorem \ref{cofincrit} and the minimal rank assumption. Moreover, $M \in \bigcap_{F \in \text{Met}_2(R)} \text{Met}_F$. Let $r_1,...,r_{c(R)}$ be a $U(R)$-basis of $R/(2)$. The map $O(M^n) \to O(M^{n+1})$ can be realised as the composition $$O(M^n) \to O(M^n \oplus \theta(r_1)) \to ... \to O(M^n \oplus (\bigoplus_{i=1}^{c(R)-1} \theta(r_i))) \to O(M^{n+1})$$ Each map above induces an isomorphism in $H_i(-,A)$ if $i \leq \frac{(n-3)c(R)-6}{2}$ by Theorem \ref{mainl}, and the present theorem immediately follows.
\end{proof}
As indicated by Example \ref{cofinexamb}, there could be cofinal forms of smaller rank than the one that the previous theorem concerns. To treat these, we will need to find larger $F$-cancellative groupoids than $\text{Met}_F$, which we can sometimes do through ad-hoc analysis of the specific case at hand. For instance:
\begin{prop}\label{dooooooo}
     Suppose that $R = \mathbb Z$ or $R$ is a finite field of characteristic $2$. For any $F \in \textup{Met}_2(R)$, the subgroupoid of $\textup{Unimod}^{s}(R)^{\simeq}$ consisting of all forms $M$ for which $g_F(M) \geq 1$ is $F$-cancellative. 
\end{prop}
\begin{proof}
         This follows from known classification theorems - over the integers, the result is that non-degenerate, indefinite, symmetric bilinear forms are classified by their rank, signature and parity (see \cite[Chp. 5, Sec. 2.2, Thm. 6]{arithmetic}), and over finite fields of characteristic $2$, the result (see \cite{ffc2} for a proof) is that there is precisely one isomorphism class of forms of odd rank, and precisely two of even rank (either a direct sum of copies of $H$ or of $[1]^{2}$). 
\end{proof}
The above result shows that one actually has homological stability with respect to stabilization by $F \in \text{Met}_2(R)$ starting with an arbitrary offset form in the case that $R$ is $\mathbb Z$ or a finite field of characteristic $2$. It is not clear to us if this holds for $R = \mathbb Z[i]$, but with some work, we can obtain larger cancellative groupoids than $\text{Met}_F$. To begin with, we observe:
\begin{lem}\label{odddiag}
    Let $F \in \textup{Met}_2(\mathbb Z[i])$, and $r \in \{1,i\}$. Then $F \oplus [r]$ is isomorphic to a diagonal form.
\end{lem}
\begin{proof}
    The cases that $F = \theta(1)$ and $F = \theta(i)$ are clear, because $\theta(1) \cong \text{diag}(1,1)$ and $\theta(i) \cong \text{diag}(i,i)$. In the case $F = \theta(0)$, there are isomorphisms $\psi_1\colon F \oplus [1] \to \text{diag}(1,1,1)$ given by $(x_1,x_2,x_3) \mapsto (x_1 + x_2+ix_3,ix_1-x_3,ix_2-x_3)$ and $\psi_i\colon F \oplus [i] \to \text{diag}(i,i,i)$ given by $(x_1,x_2,x_3) \mapsto (x_1-ix_2+ix_3,ix_1-x_3,x_2-x_3)$. In the only remaining case that $F = \theta(1+i)$, we have isomorphisms $\phi_1\colon F \oplus [1] \to \text{diag}(i,i,1)$ given by $(x_1,x_2,x_3) \mapsto (x_1-x_3,ix_1-x_2-ix_3,x_2+x_3)$ and $\phi_i \colon F \oplus [i] \to \text{diag}(1,1,i)$ given by $(x_1,x_2,x_3) \mapsto (ix_1-x_3,x_1+x_2+ix_3,x_2-x_3)$.
\end{proof}
\begin{lem}\label{representzero}
    Let $M$, $D$, and $D'$ be forms in $\textup{Unimod}^{s}({\mathbb Z[i]})^{\simeq}$ such that $D$ and $D'$ are diagonal and $M \oplus D \cong D'$. Then for any non-degenerate rank one subform $E \subseteq D$, the form $M \oplus E$ is isomorphic to a diagonal form.
\end{lem}
\begin{proof}
    We make a preliminary observation: every non-degenerate diagonal form over $\mathbb Z[i]$ is of the form $[1]^n \oplus [i]^{m}$, and $z([1]^n \oplus [i]^{m}) = \lfloor n/2 \rfloor + \lfloor m/2 \rfloor$, as can be checked using the isomorphisms $\theta(1) \cong \text{diag}(1,1)$ and $\theta(i) \cong \text{diag}(i,i)$, together with Corollary \ref{complementgenus} and the fact that the isotropic rank of the diagonal form $\text{diag}(1,i)$ equals $0$. In particular, every diagonal form of rank at least $3$ has non-zero isotropic rank.

    Choose a decomposition $M \cong M' \oplus M''$ such that $z(M') = 0$ and $M''$ is metabolic (using Corollary \ref{complementgenus}). An iterative application of Lemma \ref{odddiag} shows that the $M'' \oplus [r]$ is isomorphic to a diagonal form for $r \in \{1,i\}$. Therefore, we reduce to the case that $z(M) = 0$.

    We may assume without loss of generality that $D$ is metabolic. Corollary \ref{complementgenus} gives $z(M) = z(D') - \text{rank}(D)/2 \geq \frac{\text{rank}(D')- \text{rank}(D)}{2}-1 = \frac{\text{rank}(M)-2}{2}$, with the inequality following from the first paragraph. Thus, we only have to deal with the case $\text{rank}(M) = 2$.

    Let $E$ be a non-degenerate rank one subform of $D$, and consider the form $N \coloneq M \oplus E$. There exists a diagonal form $E'$ (for instance, the orthogonal complement of $E$ in $D$) such that $N \oplus E' \cong D'$. Since $\text{rank}(N) = 3$, the previous paragraph informs us that $z(N) = 1$. Hence, $N$ has a summand in $\text{Met}_2(R)$ and it follows that $N$ is isomorphic to a diagonal form by Lemma \ref{odddiag}.
\end{proof}
\begin{theorem}\label{cofininzi}
    Let $R = \mathbb Z[i]$ and $F = \theta(r)$ for $r \in \{1,i\}$. Let $\mathcal{D}_F$ be the full subgroupoid of $\textup{Unimod}^{s}{(\mathbb Z[i])}^{\simeq}$ consisting of $0$ and all diagonal forms $M$ with $g_F(M) \geq 1$. Then $\mathcal{D}_F$ is $F$-cancellative.
\end{theorem}
\begin{proof}
    Conditions 1, 2 and 3 of Definition \ref{cancellativeform} are immediate (recall $F \cong \text{diag}(r,r)$).

    To verify Condition 4, suppose $M \oplus F \in \mathcal{D}_F$ and $g_F(M) \geq 1$. Writing $M = N \oplus F \cong N \oplus \text{diag}(r,r)$, an application of Lemma \ref{representzero} shows that $N \oplus [r]$ is isomorphic to a diagonal form, from which it follows that $M \in \mathcal{D}_F$.

    It remains to check Condition 5. Suppose $M,N \in \mathcal{D}_F$ satisfy $M \oplus F \cong N \oplus F$. Clearly $M = 0$ if and only if $N = 0$. Otherwise, $\text{rank}(M) = \text{rank}(N)$, and since $g_F(M)$ and $g_F(N)$ are at least $1$, we also have $P(M) = P(M\oplus F) = P(N \oplus F) = P(N)$. The fact that $M \oplus F \cong N \oplus F$ also implies that both $M$ and $N$ have the same discriminant. The maps in the proof of Lemma \ref{odddiag} furnish an isomorphism $\text{diag}(1,1,1,i) \cong \theta(0) \oplus [i] \oplus [1] \cong \text{diag}(i,i,i,1)$, whence it follows that diagonal forms of the same parity, rank, and discriminant are isomorphic. Thus $M \cong N$.
\end{proof}
\begin{kor}\label{useintro} Let $R = \mathbb Z[i]$. For constant coefficients ${A}$, the map $$H_i(O(\textup{diag}(1,i)^n),{A}) \to H_i(O(\textup{diag}(1,i)^{n+1}),{A})$$ is an isomorphism if $i \leq \frac{n-12}{2}$.
\end{kor}
\begin{proof}
 Let $M$ be the form $\text{diag}(1,i)^2$; it is the cofinal metabolic form of minimal rank (see Example \ref{cofinexamb}). Theorem \ref{homstabcofin} tells us that the map $H_i(O(M^n),A) \to H_i(O(M^{n+1}),A)$ is an isomorphism if $i \leq n-6$. The key observation is that both $M^n$ and $M^{n} \oplus \text{diag}(1,i)$ are objects of the $F$-cancellative groupoid $\mathcal{D}_F$ (Theorem \ref{cofininzi}), where $F = \theta(1)$ or $\theta(i)$. Thus, the stabilisation-by-$M$ map, which can be written as a composite of stabilisations by $\theta(1)$ and $\theta(i)$, $$H_i(O(M^n \oplus \text{ diag}(1,i)),A) \to H_i(O(M^n \oplus \text{ diag}(1,i) \oplus \theta(1)),A)\to H_i(O(M^{n+1} \oplus \text{ diag}(1,i)),A)$$ is an isomorphism if $i \leq n-6$. Consequently, in the sequence 
\[\begin{tikzcd}
	{H_i(O(\text{diag}(1,i)^{2n}),A)} & {H_i(O(\text{diag}(1,i)^{2n+1}),A)} \\
	& {H_i(O(\text{diag}(1,i)^{2n+2}),A)} & {H_i(O(\text{diag}(1,i)^{2n+3}),A)}
	\arrow[from=1-1, to=1-2]
	\arrow[from=1-2, to=2-2]
	\arrow[from=2-2, to=2-3]
\end{tikzcd}\] where each arrow is the stabilisation-by-$\text{diag}(1,i)$ map, we gather that the composites of the first two and the last two arrows are isomorphisms if $i \leq n -6$. Therefore, each arrow is an isomorphism by the $2$-out-of-$6$ property, proving the corollary.
\end{proof}
One can similarly deduce, using Proposition \ref{dooooooo} in place of Theorem \ref{cofininzi}, homological stability with respect to stabilisation by the form $[1]$ in the case that $R$ is a finite field of characteristic $2$.

We conclude with calculations of the low-dimensional cohomology of the cofinal form $\theta(1)$ for $R = \mathbb Z$, as stated in the introduction (Corollary \ref{calculationdintro}). Let $\textup{O}_{\langle n,n \rangle}(\mathbb Z)$ denote the group of isometries of the form $\theta(1)^n$.
\begin{kor}\label{calculationd}
    For an odd regular prime $p$, there are maps $$\mathbb F_2[v_i,w_i,a_i|i \geq 1] \to H^{*}(\textup{O}_{\langle n,n \rangle}(\mathbb Z),\mathbb{F}_2)$$ $$\mathbb F_p[p_i^{+},p_i|i \geq 1] \otimes_{\mathbb F_p} \Lambda_{\mathbb F_p}[y_i|i \geq 1] \to H^{*}(\textup{O}_{\langle n,n \rangle}(\mathbb Z),\mathbb{F}_p)$$ where the $v_i,w_i$ have degree $i$, the $p_i^{+}$ have degree $4i$, the $p_i$ have degree $2(p-1)i$, the $y_i$ have degree $(p-1)(2i-1)$, and the $a_i$ have degree $2i-1$, that are isomorphisms in degree $\leq \frac{n-9}{2}$. 
\end{kor}
\begin{proof}
    Both assertions follow from Theorem \ref{homstabcofin} coupled with calculations in other work - the first from \cite[Theorem 8.1.9]{hebestreit2025stablemodulispaceshermitian} and the second from upcoming work of Hebestreit, Land, and Nikolaus (\cite{upcomingcalc}).
\end{proof}
\printbibliography
\end{document}